\documentclass[10pt]{article}
\usepackage[latin1]{inputenc}
\usepackage{epsfig}
\usepackage{color}
\usepackage[british,english]{babel}
\usepackage{amsthm}
\usepackage{amsmath}
\usepackage{amsfonts}
\usepackage{amssymb}
\usepackage{graphicx}
\newtheorem{corollary}{Corollary}[section]

\newtheorem{lemma}[corollary]{Lemma}
\newtheorem{proposition}[corollary]{Proposition}
\newtheorem{remark}[corollary]{Remark}
\newtheorem{theorem}[corollary]{Theorem}
\newfont{\sBlackboard}{msbm10 scaled 900}

\newcommand{\mylabel}[1]{\label{#1}
            \ifx\undefined\stillediting
            \else \fbox{$#1$}\fi }
\newcommand{\BE}{\begin{equation}}

\newcommand{\EEQ}{\end{equation}}
\newcommand{\rfb}[1]{\mbox{\rm
   (\ref{#1})}\ifx\undefined\stillediting\else:\fbox{$#1$}\fi}

\newfont{\Blackboard}{msbm10 scaled 1200}

\newfont{\roma}{cmr10 scaled 1200}

\def\CC{\rm \hbox{C\kern-.56em\raise.4ex
         \hbox{$\scriptscriptstyle |$}\kern+0.5 em }}

\def\n{|\kern -.05cm{|}\kern -.05cm{|}}

\def \noame{\noalign{\medskip}}
\newcommand{\mm}    {{\hbox{\hskip 0.5pt}}}

\newcommand{\bluff} {{\hbox{\raise 15pt \hbox{\mm}}}}

\newcommand{\ep}   {\varepsilon}
\usepackage{fancyhdr}

\makeatletter
\def\section{\@startsection {section}{1}{\z@}{-3.5ex plus -1ex minus
    -.2ex}{2.3ex plus .2ex}{\large\bf}}
\makeatother
\def\be{\begin{equation}}
\def\ee{\end{equation}}

\date{ }
\begin{document}
\thispagestyle{empty}
\title{\Large \bf  Roughness-induced effects on the thermomicropolar fluid flow through a thin domain}\maketitle
\vspace{-2cm}
\begin{center}
Igor PA${\rm \check{Z}}$ANIN\footnote{Department of Mathematics, Faculty of Science, University of Zagreb, Bijeni${\rm \check{c}}$ka 30, 10000 Zagreb (Croatia) pazanin@math.hr} and Francisco Javier SU\'AREZ-GRAU\footnote{Departamento de Ecuaciones Diferenciales y An\'alisis Num\'erico. Facultad de Matem\'aticas. Universidad de Sevilla. 41012-Sevilla (Spain) fjsgrau@us.es}
 \end{center}

 \renewcommand{\abstractname} {\bf Abstract}
\begin{abstract}
In this paper, we study the asymptotic behavior of the thermomicropolar fluid flow through a thin channel with rough boundary. The flow is governed by the prescribed pressure drop between the channel's ends and the heat exchange through the rough wall is allowed. Depending on the limit of the ratio between channel's thickness and the wavelength of the roughness, we rigorously derive different asymptotic models clearly showing the roughness-induced effects on the average velocity and microrotation. To accomplish that, we employ the adaptation of the unfolding method to a thin-domain setting.
\end{abstract}
\bigskip\noindent

\noindent {\small \bf AMS classification numbers:}  35B27, 35Q35, 76A05, 76M50.  \\

\ \\
\section {Introduction}\label{S1}

The model of micropolar fluid, proposed by Eringen \cite{eringen} has been extensively studied both in the engineering and mathematical literature, due to its practical importance. Being able to take into consideration the microstructure of the fluid particles and capture the effects of its rotation, the micropolar fluid model describes the motion of numerous real fluids better than the classical Navier-Stokes equations. Liquid crystals, animal blood, muddy fluids, certain polymeric fluids or even water in models with small scales are the typical examples. The rotation of the fluid particles is mathematically described by introducing the microrotation field (along with the standard velocity and pressure fields) and, accordingly, a new governing equation coming from the conservation of angular momentum. The model of thermomicropolar fluid, introduced also by Eringen \cite{eringenII}, represents an essential ge\-neralization of the micropolar fluid model acknowledging the variations of the fluid temperature as well. In such, non-isothermal, regime, the micropolar equations are being coupled with the heat conduction equation leading to a very complex system of PDEs. In particular, the 2D system describing the steady-state flow of incompressible, isotropic, thermomicropolar fluid flow between two horizontal plates in dimensionless form reads as follows (see e.g.~\cite{eringenIII}, \cite{Luka1},\cite{Tarasinka}):
\begin{equation}\label{system_1intro}
\left\{\begin{array}{rl}
\displaystyle {1\over Pr}(({\bf u} \cdot\nabla){\bf u} +\nabla p )=\Delta {\bf u} + {N\over 1-N}(2\nabla^\perp w +\Delta{\bf u} )+Ra\,T {\bf e}_2+{\bf f},  &\\
\\
{\rm div}( {\bf u})=0,& \\
\\
\displaystyle {M\over Pr}({\bf u} \cdot \nabla w )=L\Delta w +{2N\over 1-N}({\rm rot}({\bf u} )-2w )+g ,& \\
\\
{\bf u} \cdot \nabla T =\Delta T +D\nabla^\perp w \cdot \nabla T.&
\end{array}\right.
\end{equation}

\noindent
In the above system, the velocity vector field is denoted by ${\bf u}$, the pressure by $p$, $w$ represents the microrotation and $T$ is the temperature of the fluid. The external sources of linear and angular momentum are given by the functions ${\bf f}=(f_1,f_2)$ and $g$, respectively. We denote by ${\bf e}_ 2= (0, 1)\in \mathbb{R}^2$ the unit upward vector, whereas the positive constants appearing in (\ref{system_1intro}) represent the following (see e.g.~\cite{Luka}):
\begin{itemize}
\item $N$ is the coupling parameter, i.e. the relation between the Newtonian and microrotation viscosities,
\item $M$ is the relation between the moment of inertia and geometry,
\item $L$ is the couple stress parameter, i.e. the relation between the geometry and the properties of the fluid,
\item $D$ is the micropolar heat conduction parameter, i.e. the relation between the micropolar thermal conduction and the geometry,
\item $Pr$ is the Prandtl number, i.e. the relation between the kinematic viscosity and the thermal diffusivity,
\item $Ra$ is the Rayleigh number, i.e. the relation between the coefficients of thermal expansion and conductivity and the geometry.
\end{itemize}
\noindent
Throughout the mathematical literature, one can find many papers on the rigorous derivation of the asymptotic models describing the isothermal flow of a micropolar fluid, see e.g.~\cite{benpazrad}, \cite{Bonn_Paz_SG2}, \cite{Bonn_Paz_SG}, \cite{duppanstavI}, \cite{duppanstavII}, \cite{pazII}, \cite{Pazanin_SG}. Although there have been a number of recent papers concerning engineering applications of the thermomicropolar fluid model (see e.g.~\cite{cheng}, \cite{hosmanroyhos}, \cite{hosmanroyhosII}, \cite{rahelt}), the rigorous treatments for such models are very sparse.  Most recently, the system (\ref{system_1intro}) has been studied in \cite{Pazanin_thermomicropolar} for the thermicropolar flow through a thin channel with smooth walls, namely:
$$\Omega^\ep=\left\{(x_1,x_2)\in\mathbb{R}^2\,:\, x_1\in\omega,\ 0<x_2<\varepsilon\right\},\,\,\,\,\omega=(-1/2,1/2).$$
The flow is assumed to be governed by the prescribed
pressure drop between channel's ends, given by $q_{-1/2}$ and $q_{1/2}$, and the heat exchange between the fluid inside the channel and the exterior medium is allowed through the upper wall by using Newton's cooling law. Using the asymptotic analysis with respect to the thickness of the channel, a higher-order
asymptotic solution has been rigorously derived. In particular, assuming that $f_1$ and $g$ only depends on the horizontal variable and after a dilatation in the vertical variable, it is proved that the average velocity ${\bf U}^{av}=(U_1^{av},U_2^{av})$ and the microrotation $W^{av}$ at the main-order term are respectively given by:
\begin{equation}\label{av_pazanin_intro}
\begin{array}{l}
\displaystyle U^{av}_1={1\over 12}{1-N\over Pr}\left(q_{-1/2}-q_{1/2}+Pr\int_{-1/2}^{1/2}f_1(\xi)\,d\xi\right),\quad U^{av}_2=0,\quad   W^{av}={1\over 12 }{1\over L} g(x_1),\quad\hbox{in }\omega.
\end{array}
\end{equation}
Moreover, the explicit expressions for the pressure approximation is obtained and for the average of the temperature as well, acknowledging the effects of fluid's microstructure through the presence of the couple stress parameter $L$  and the micropolar heat conduction parameter $D$.\\
\ \\
\noindent
In great majority of the applications, the domain boundaries are not perfectly smooth, i.e. they contain some irregularities. Thus, in the present paper, we aim to generalize the results from \cite{Pazanin_thermomicropolar} to a case of a thin channel with an irregular upper wall described by
$$x_2=\eta_\ep h\left({x_1\over \ep}\right),$$
where $\eta_\ep$ is the thickness of the roughness, $\ep$ is the period of the roughness and $h$ is a positive and periodic function (see Section \ref{sec:setting}).  This kind of thin rough domain has been extensively studied for the isothermal flows, see \cite{Bayada_Chambat}, \cite{Mikelic2}  for the classical Newtonian fluid flow,  \cite{Anguiano_SG} for the flow of the generalized Newtonian fluid and \cite{grauMicRough} for the micropolar fluid flow. In these papers, a critical size has been found between the thickness of the domain $\eta_\ep$ and the period of the roughness $\ep$, which is given by
$$\lambda=\lim_{\ep\to 0}{\eta_\ep\over \ep}\in [0,+\infty].$$
The critical case, $\lambda\in (0,+\infty)$, corresponds to the case in which the thickness and period of the roughness are proportional. The subcritical case, $\lambda=0$, corresponds to a very smooth roughness, and the supercritical case, $\lambda=+\infty$, corresponds to the case of a highly oscillating boundary. \\

\noindent
As far as the authors know, the flow of a thermomicropolar fluids has not been yet considered in the above described setting. The supercritical case, due to the highly oscillating boundary, leads to the conclusion that the velocity and microrotation are zero in the roughness zone (see e.g.~\cite{grauMicRough}) so, in the sequel, we study the asymptotic behavior of the solution in the critical and the subcritical case. By applying reduction of dimension techniques together with an adaptation of the unfolding method (see Section \ref{sec:estimates}) to capture the microgeometry of the roughness, depending on the relation of $\ep$ and $\eta_\ep$, we rigorously derive the following expressions for the average velocity and microrotation:
\begin{equation}\label{av_paper_intro}
\begin{array}{l}
\displaystyle U^{av}_1=a_\lambda{1-N\over Pr}\left(q_{-1/2}-q_{1/2}+Pr\int_{-1/2}^{1/2}f_1(\xi)\,d\xi\right),\quad U^{av}_2=0,\quad   W^{av}=b_\lambda{1\over L} g(x_1),\quad\hbox{in }\omega.
\end{array}
\end{equation}
where $a_\lambda, b_\lambda\in \mathbb{R}^+$ are obtained through local problems depending on the value of $\lambda\in[0,+\infty)$ and give the roughness-induced effects on the velocity and microrotation. In the critical case $\lambda\in (0,+\infty)$, the parameters $a_\lambda, b_\lambda$ are computed through local PDE problems (see Section \ref{sec:critical}, Theorem \ref{mainthmcrit}). However, in the subcritical case $\lambda=0$, the parameters $a_0, b_0$ can be explicitly computed (see Section \ref{sec:subcritical}, Theorem \ref{mainthmsubcrit}). In both cases, we obtain the same expression for the pressure as in \cite{Pazanin_thermomicropolar}. Moreover, the average of the temperature is obtained through a nonlinear problem in the critical case and is explicitly given in the subcritical case. Since the obtained findings are amenable for the numerical simulations, we believe that it could prove useful in the engineering practice as well.

\section{Formulation of the problem and preliminaries}\label{sec:setting}
In this section, we first define the thin, rough domain and some sets necessary to study the asymptotic behavior of the solutions.  Next, we introduce the problem considered in the thin domain and also, the rescaled problem posed in the domain of fixed height.

\subsection{The domain and some notation}  Let us denote $\omega=(-1/2,1/2)\subset\mathbb{R}$. We consider a thin domain with a rapidly oscillating thickness defined by
\begin{equation}\label{Omegaep}
\Omega^\varepsilon=\{x=(x_1,x_2)\in\mathbb{R}^2\,:\, x_1\in \omega,\ 0<x_2< h_\varepsilon(x_1)\},
\end{equation}
Here, the function $h_\ep(x_1)= \eta_\varepsilon h\left(x_1/\varepsilon\right)$ represents the real gap between the two surfaces. The small parameter $\eta_\varepsilon$ is related to the film thickness and  the small parameter $\varepsilon$ is the wavelength of the roughness. Here, we consider that $\eta_\varepsilon$ is of order smaller or equal than $\varepsilon$, i.e. we consider
\begin{equation}\label{cases}
\eta_\varepsilon\approx\varepsilon \quad \hbox{or}\quad \eta_\varepsilon\ll\varepsilon.
\end{equation}
\noindent Function $h\in W^{1,\infty}(\mathbb{R})$, $Z'$-periodic with $Z'=(-1/2,1/2)$ the cell of periodicity in $\mathbb{R}$, and there exist $h_{\rm min}$ and $h_{\rm max}$ such that
$$0<h_{\rm min}=\min_{z_1\in Z'} h(z_1),\quad h_{\rm max}=\max_{z_1\in Z'}h(z_1)\,.$$

\noindent
We define the boundaries of $\Omega^\varepsilon$ as follows
$$\begin{array}{c}
\displaystyle \Gamma_0 =\left\{(x_1,x_2)\in\mathbb{R}^2\,:\, x_1\in \omega,\ x_2=0\right\},\quad \Gamma_1^\varepsilon=\left\{(x_1,x_2)\in\mathbb{R}^2\,:\, x_1\in \omega,\ x_2=  h_\varepsilon(x_1)\right\}\\
\\
\displaystyle\Sigma_i^\varepsilon=\left\{(x_1,x_2)\in\mathbb{R}^2\,:\, x_1=i,\ 0<x_2<  h_\varepsilon(x_1)\right\},\quad i=-1/2,1/2.
\end{array}$$
 \noindent We also define the respective rescaled sets  $$\widetilde \Omega^\varepsilon=\omega\times (0,h(x_1/\varepsilon)), \quad \widetilde\Gamma_1^\varepsilon=\omega\times \{h(x_1/\varepsilon)\}\quad\hbox{and}\quad \widetilde \Sigma_i^\varepsilon=\{i\}\times (0,h(i/\varepsilon)),\quad i=-1/2,1/2.$$
Related  to the microstructure of the periodicity of the  boundary,  we consider that the domain $\omega$ is divided by a mesh of size $\ep$: for $k'\in\mathbb{Z}$, each cell $Z'_{k',\ep}=\ep k'+\ep Z'$. We define $T_\ep=\{k'\in\mathbb{Z}\,:\, Z'_{k',\ep}\cap\omega\neq \emptyset\}$. In this setting, there exists an exact finite number of periodic sets $Z'_{k',\ep}$ such that $k'\in T_\ep$.   Also, we define $Z_{k',\ep}=Z'_{k',\ep}\times (0,h(z_1))$ and  $Z=Z'\times (0,h(z_1))$, which is the reference cell in $\mathbb{R}^2$. We define the boundaries $\hat \Gamma_0=Z'\times \{0\}$, $\hat \Gamma_1=Z'\times\{h(z_1)\}$,  $\hat \Sigma_i =\{i\}\times \{i\}\times (0,h(i))$, $ i=-1/2,1/2$.  The quantity   $h_{\rm max}$ allows us to define the extended sets  $\Omega=\omega\times  (0, h_{\rm max})$    and $\Gamma_1=\omega\times \{h_{\rm max}\}$.\\

 \noindent In order to apply the unfolding method, we will use the following notation. For $k'\in \mathbb{Z}$, we define $\kappa: \mathbb{R}\to \mathbb{Z}$ by
\begin{equation}\label{kappa_fun}
\kappa(x_1)=k' \Longleftrightarrow x_1\in Z_{k',1}\,.
\end{equation}
Remark that $\kappa$ is well defined up to a set of zero measure in $\mathbb{R}$ (the set $\cup_{k\in \mathbb{Z}}\partial Y_{k,1}$). Moreover, for every $\varepsilon>0$, we have
$$\kappa\left({x_1\over \varepsilon}\right)=k'\Longleftrightarrow x_1\in Z_{k',\varepsilon}\,.$$

\noindent We denote by $C$ a generic constant which can change from line to line.\\

\noindent We use the following notation for the partial differential operators:
 $$\begin{array}{c}
 \displaystyle \Delta \Phi^\varepsilon=\left({\partial^2\Phi_1^\varepsilon\over \partial x_1^2}+{\partial^2\Phi_1^\varepsilon\over \partial x_2^2}\right){\bf e}_1+\left({\partial^2\Phi_2^\varepsilon\over \partial x_1^2}+{\partial^2\Phi_2^\varepsilon\over \partial x_2^2}\right){\bf e}_2,\quad  \displaystyle\Delta\varphi^\varepsilon={\partial^2\varphi^\varepsilon\over \partial x_1^2}+{\partial^2\varphi^\varepsilon\over \partial x_2^2},\\
 \\
  \displaystyle {\rm div}({\Phi}^\varepsilon)={\partial \Phi_1^\varepsilon\over \partial x_1}+{\partial \Phi_2^\varepsilon\over \partial x_2},\quad {\rm rot}({\Phi}^\varepsilon)={\partial \Phi_2^\varepsilon\over \partial x_1}-{\partial \Phi_1^\varepsilon\over \partial x_2}, \quad
   \displaystyle \nabla^\perp\varphi^\varepsilon=\left({\partial\varphi^\varepsilon\over \partial x_2},-{\partial\varphi^\varepsilon\over \partial x_1}\right),
 \end{array}$$
where ${\Phi}^\varepsilon=(\Phi^\varepsilon_1, \Phi_2^\varepsilon)$ is a vector function and $\varphi^\varepsilon$ is a scalar function defined in $\Omega^\varepsilon$. \\

\noindent Moreover,   for   ${\tilde \Phi}^\varepsilon=(\tilde \Phi^\varepsilon_1, \tilde \Phi_{2}^\varepsilon)$  a vector function and $\tilde \varphi^\varepsilon$  a scalar function defined in $\widetilde\Omega^\varepsilon$, after a dilatation in the vertical variable, we will use the following operators
 $$\begin{array}{ll}
 \displaystyle \Delta_{\eta_\varepsilon}{\tilde \Phi}^\varepsilon=\left({\partial^2\tilde \Phi_1^\varepsilon\over \partial x_1^2}+{1\over \eta_\varepsilon^2}{\partial^2\tilde \Phi_1^\varepsilon\over \partial {z_2}^2}\right){\bf e}_1+\left({\partial^2\tilde \Phi_{2}^\varepsilon\over \partial x_1^2}+{1\over \eta_\varepsilon^2}{\partial^2\tilde \Phi_{2}^\varepsilon\over \partial {z_2}^2}\right){\bf e}_2,&  \displaystyle\Delta_{\eta_\varepsilon}\tilde \varphi^\varepsilon={\partial^2\tilde \varphi^\varepsilon\over \partial x_1^2}+{1\over \eta_\varepsilon^2}{\partial^2\tilde \varphi^\varepsilon\over \partial {z_2}^2},\\
 \\
  \displaystyle {\rm div}_{\eta_\varepsilon}({\tilde \Phi}^\varepsilon)={\partial \tilde \Phi_1^\varepsilon\over \partial x_1}+{1\over \eta_\varepsilon}{\partial\tilde  \Phi_{2}^\varepsilon\over \partial {z_2}},\quad {\rm rot}_{\eta_\varepsilon}({\tilde  \Phi}^\varepsilon)={\partial \tilde \Phi_{2}^\varepsilon\over \partial x_1}-{1\over \eta_\varepsilon}{\partial \tilde \Phi_1^\varepsilon\over \partial z_2}, &
   \displaystyle \nabla_{\eta_\varepsilon}^\perp\tilde \varphi^\varepsilon=\left({1\over \eta_\varepsilon}{\partial\tilde \varphi^\varepsilon\over \partial z_2},-{\partial\tilde \varphi^\varepsilon\over \partial x_1}\right).
 \end{array}$$

\noindent  For $\varphi=(\varphi_1,\varphi_2)$ and $\psi=(\psi_1,\psi_2)$, we define $\tilde\otimes$ by
 \begin{equation}\label{tildeotimes}
 (\varphi\tilde\otimes \psi)_{ij}=\varphi_i\psi_j,\quad i=1,\ j=1,2\,.
 \end{equation}

 \noindent Finally, we introduce some functional spaces. $L^{q}_0$  is the space of functions of $L^q$  with zero mean value. Let $C^\infty_{\#}(Z)$ be the space of infinitely differentiable functions in $\mathbb{R}^3$ that are $Z'$-periodic. By $L^q_{\#}(Z)$ (resp. $W^{1,q}_{\#}(Z)$), $1<q<+\infty$, we denote its completion in the norm $L^q(Z)$ (resp. $W^{1,q}(Z)$) and by $L^q_{0,\#}(Z)$  the space of functions in $L^q_{\#}(Z)$ with zero mean value.

\subsection{The problem}  The governing equations in dimensionless form are given by
\begin{equation}\label{system_1}
\left\{\begin{array}{rl}
\displaystyle {1\over Pr}(({\bf u}^\varepsilon\cdot\nabla){\bf u}^\varepsilon+\nabla p^\varepsilon)=\Delta {\bf u}^\varepsilon+ {N\over 1-N}(2\nabla^\perp w^\varepsilon+\Delta{\bf u}^\varepsilon)+Ra\,T^\varepsilon{\bf e}_2+{\bf f}^\varepsilon & \hbox{in}\ \Omega^\varepsilon,\\
\\
{\rm div}( {\bf u}^\varepsilon)=0& \hbox{in}\ \Omega^\varepsilon,\\
\\
\displaystyle {M\over Pr}({\bf u}^\varepsilon\cdot \nabla w^\varepsilon)=L\Delta w^\varepsilon+{2N\over 1-N}({\rm rot}({\bf u}^\varepsilon)-2w^\varepsilon)+g^\varepsilon& \hbox{in}\ \Omega^\varepsilon,\\
\\
{\bf u}^\varepsilon\cdot \nabla T^\varepsilon=\Delta T^\varepsilon+D\nabla^\perp w^\varepsilon\cdot \nabla T^\varepsilon& \hbox{in}\ \Omega^\varepsilon.
\end{array}\right.
\end{equation}
We complete the above system with the following boundary conditions on the bottom
\begin{equation}\label{BCBot}
\begin{array}{ll}
\displaystyle {\bf u}^\varepsilon=0,\quad w^\varepsilon=0,\quad T^\varepsilon=0&\hbox{on}\ \Gamma_0,
\end{array}
\end{equation}
the following conditions on the lateral boundaries
\begin{equation}\label{BClat}
\begin{array}{rl}
\displaystyle {\bf u}^\varepsilon\cdot {\bf e}_2=0,\quad w^\varepsilon=0,\quad T^\varepsilon=0,\quad p^\varepsilon={1\over \eta_\varepsilon^2}q_i&\hbox{on}\ \Sigma_{i}^\varepsilon,\ i=\{-1/2,1/2\},
\end{array}
\end{equation}
and the following boundary conditions on the top boundary
\begin{equation}\label{BCTopBot1}
\begin{array}{rl}
{\bf u}^\varepsilon=0,\quad  w^\varepsilon=0,\quad  \nabla T^\varepsilon \cdot {\bf n}=Nus(G^\varepsilon-T^\varepsilon)&\hbox{on}\ \Gamma_1^\varepsilon.
\end{array}
\end{equation}
Here, ${\bf u}^\varepsilon=(u^\varepsilon_1, u^\varepsilon_2)$ represents the velocity field, $p^\varepsilon$ the pressure, $w^\varepsilon$ the microrotation and $T^\varepsilon$  the temperature. The external body forces are given by ${\bf f}^\varepsilon=(f^\varepsilon_1, f^\varepsilon_2)$ and the external body torque by $g^\varepsilon$.
\\

\noindent We make the following assumptions:
\begin{itemize}

\item[--] The Robin boundary condition (\ref{BCTopBot1})$_3$ comes from Newton's cooling law and describes the heat exchange through the upper wall between the exterior medium and the fluid inside the channel. Due to domain's microstructure, it is assumed that the exterior temperature $G^\varepsilon=G(x_1/\varepsilon)$ with $G\in L^2(Z')$ a given $Z'$-periodic and bounded function
depending only on the longitudinal variable. \\

\item[--] Following previous result \cite{Pazanin_thermomicropolar}, we consider that the Nusselt number $Nus$ depends on $\varepsilon$, whereas all the other characteristic numbers are kept independent of $\varepsilon$. In fact, we compare the Nusselt number $Nus$ to the small parameter of the height $\eta_\varepsilon$. Namely, we consider the following scaling of the Nusselt number
\begin{equation}\label{Nusselt}Nus=\eta_\varepsilon\,k,\quad k=\mathcal{O}(1).
\end{equation}
\item[--] We assume that the external source functions are independent of the variable $x_2$ and take the following scaling
\begin{equation}\label{externalforces}
{\bf f}^\varepsilon={1\over \eta_\varepsilon^2}(f_1(x_1),0),\quad   g^\varepsilon={1\over \eta_\varepsilon^2}g(x_1),\quad \hbox{with}\quad f_1, g\in L^2(\omega).
\end{equation}
\end{itemize}
Under the previous assumptions, the well-posedness of the problem (\ref{system_1})-(\ref{BCTopBot1}) can be established using the methods from \cite{Luka} (see also \cite{kalluksie}) and prove that there exists a unique weak solution $({\bf u}^\varepsilon, w^\varepsilon, p^\varepsilon, T^\varepsilon)\in H^1(\Omega^\varepsilon)^2\times H^1_0(\Omega^\varepsilon)\times L^2_0(\Omega^\varepsilon)  \times H^1(\Omega^\varepsilon)$.
 \\

\noindent Our aim is to study the asymptotic behavior of $u_\ep$, $w_\ep$, $p_\ep$ and $T^\varepsilon$ when $\ep$ and $\eta_\ep$ tend to zero and identify homogenized models coupling the effects of the thickness of the domain and the roughness of the boundary.  For this, we use the dilatation in the variable $x_2$ given by
\begin{equation}\label{dilatacion}
z_2={x_2\over \eta_\ep}\,,
\end{equation}
in order to have the functions defined in the open set with fixed height   $\widetilde\Omega_\ep$ and the rescaled boundaries $\widetilde \Gamma_1^\varepsilon$ and  $\widetilde \Sigma^\varepsilon_i$, $i=-1/2,1/2$. Then, using the change of variables (\ref{dilatacion}) in  (\ref{system_1})-(\ref{BCTopBot1}), we obtain the following rescaled system
\begin{equation}\label{system_1_dil}
\left\{\begin{array}{rl}
\displaystyle {1\over Pr}((\tilde {\bf u}^\varepsilon\cdot\nabla_{\eta_\varepsilon})\tilde{\bf  u}^\varepsilon+\nabla_{\eta_\varepsilon}\tilde  p^\varepsilon)=\Delta_{\eta_\varepsilon} \tilde {\bf u}^\varepsilon+ {N\over 1-N}(2\nabla_{\eta_\varepsilon}^\perp \tilde w^\varepsilon+\Delta_{\eta_\varepsilon}\tilde{\bf  u}^\varepsilon)+Ra\,\tilde T^\varepsilon{\bf e}_2+{\bf f}^\varepsilon& \hbox{in}\ \widetilde \Omega^\varepsilon,\\
\\
{\rm div}_{\eta_\varepsilon}( \tilde{\bf  u}^\varepsilon)=0& \hbox{in}\ \widetilde \Omega^\varepsilon,\\
\\
\displaystyle {M\over Pr}(\tilde{\bf  u}^\varepsilon\cdot \nabla_{\eta_\varepsilon} \tilde w^\varepsilon)=L\Delta_{\eta_\varepsilon} \tilde w^\varepsilon+{2N\over 1-N}({\rm rot}_{\eta_\varepsilon}
( \tilde{\bf  u}^\varepsilon)-2\tilde w^\varepsilon)+g^\varepsilon& \hbox{in}\ \widetilde \Omega^\varepsilon,\\
\\
\tilde{\bf  u}^\varepsilon\cdot \nabla_{\eta_\varepsilon} \tilde T^\varepsilon=\Delta_{\eta_\varepsilon} \tilde T^\varepsilon+D\nabla_{\eta_\varepsilon}^\perp \tilde w^\varepsilon\cdot \nabla_{\eta_\varepsilon} \tilde T^\varepsilon& \hbox{in}\ \widetilde \Omega^\varepsilon,
\end{array}\right.
\end{equation}
with the boundary conditions
\begin{equation}\label{BCbot1_dil}
\begin{array}{ll}
\displaystyle \tilde{\bf  u}^\varepsilon=0,\quad \tilde w^\varepsilon=0,\quad \tilde T^\varepsilon=0&\hbox{on}\   \Gamma_0,\end{array}
\end{equation}
\begin{equation}\label{BClat1_dil}
\begin{array}{ll}
\displaystyle \tilde{\bf  u}^\varepsilon\cdot {\bf e}_2=0,\quad \tilde w^\varepsilon=0,\quad \tilde T^\varepsilon=0,\quad \tilde p^\varepsilon={1\over \eta_\varepsilon^2}q_i&\hbox{on}\ \widetilde \Sigma^\varepsilon_i,\ i=\{-1/2,1/2\},\end{array}
\end{equation}
\begin{equation}\label{BCTop1_dil}
\begin{array}{rl}
\tilde{\bf  u}^\varepsilon=0,\quad  \tilde w^\varepsilon=0,\quad  \nabla_{\eta_\varepsilon}\tilde T^\varepsilon \cdot {\bf n}=\eta_\ep k(G^\varepsilon-\tilde T^\varepsilon)&\hbox{on}\ \widetilde \Gamma_1^\varepsilon,
\end{array}
\end{equation}

\noindent The unknown functions in the above system are given by $\tilde{\bf  u}^\varepsilon(x_1,z_2)={\bf u}^\varepsilon(x_1,\eta_\varepsilon z_2)$, $\tilde p^\varepsilon(x_1,z_2)=p^\varepsilon(x_1,\eta_\varepsilon z_2)$, $\tilde w^\varepsilon(x_1,z_2)=w^\varepsilon(x_1,\eta_\varepsilon z_2)$ and $\tilde T^\varepsilon(x_1,z_2)=T^\varepsilon(x_1,\eta_\varepsilon z_2)$ for a.e. $(x_1,z_2)\in \widetilde\Omega^\varepsilon$. \\

\noindent  Our goal then is to describe the asymptotic behavior of this new sequences $ \tilde{\bf  u}_\ep$, $\tilde w_\ep$, $\tilde p_\ep$ and $\tilde T^\varepsilon$ when $\ep$ and $\eta_\ep$ tend to zero.  To do this, we establish the a priori estimates and introduce the adaptation of the unfolding method in Section \ref{sec:estimates}. We obtain the limit model of the critical case ($\eta_\ep\approx \ep$) in Section \ref{sec:critical} and of the sub-critical case ($\eta_\ep\ll \ep$) in Section \ref{sec:subcritical}.

\section{{\it A priori} estimates} \label{sec:estimates}
This section is devoted to derive the {\it a priori} estimates of the unknowns and is divided in three parts. First, we deduce  the {\it a priori} estimates for velocity, microrotation and temperature and second, we derive the estimates for pressure. Finally, we introduce the adaptation of the unfolding method and derive the {\it a priori} estimates of the unfolded functions.

\subsection{Estimates for velocity, microroration and temperature} To derive the desired estimates,  let us recall some well-known technical results (see, e.g. \cite{Pazanin_thermomicropolar}).
\begin{lemma}[Poincar\'e and Ladyzhenskaya inequalities]\label{Poincare_lemma} For all $\varphi\in H^1(\Omega^\varepsilon)$ such that $\varphi=0$ on $\Gamma_0$, there hold the following inequalities
\begin{equation}\label{Poincare}
\|\varphi\|_{L^2(\Omega^\varepsilon)}\leq C\eta_\varepsilon\|\nabla \varphi\|_{L^2(\Omega^\varepsilon)^2},\quad
 \|\varphi\|_{L^4(\Omega^\varepsilon)}\leq C\eta_\varepsilon^{1\over 2}\|\nabla \varphi\|_{L^2(\Omega^\varepsilon)^2}.
\end{equation}
Moreover, from the change of variables (\ref{dilatacion}),   there hold the following rescaled estimates
\begin{equation}\label{Poincare2}
\|\tilde \varphi\|_{L^2(\widetilde \Omega^\varepsilon)}\leq C\eta_\varepsilon\|\nabla_{\eta_\ep} \tilde \varphi\|_{L^2(\widetilde \Omega^\varepsilon)^2},\quad \|\tilde \varphi\|_{L^4(\widetilde \Omega^\varepsilon)}\leq C\eta_\varepsilon^{3\over 4}\|\nabla_{\eta_\ep} \tilde \varphi\|_{L^2(\widetilde \Omega^\varepsilon)^2}.
\end{equation}
\end{lemma}
\begin{lemma}
[Trace estimates]\label{trace_lemma} For all $\varphi\in H^1(\Omega^\varepsilon)$ such that $\varphi=0$ on $\Gamma_0$, there hold the following trace estimates:
\begin{equation}\label{trace_estim1}\|\varphi\|_{L^2(\Gamma_1^\varepsilon)}\leq C\eta_\varepsilon^{1\over 2}\|\nabla\varphi\|_{L^2(\Omega^\varepsilon)^2},
\end{equation}
Moreover, for every case,  the rescaled function satisfies the following estimate
\begin{equation}\label{trace_estim1_dil}\|\tilde \varphi\|_{L^2(\widetilde \Gamma_1^\varepsilon)}\leq C\eta_\varepsilon\varepsilon^{-{1\over 2} }\|\nabla_{\eta_\ep}\tilde \varphi\|_{L^2(\widetilde \Omega^\varepsilon)^2},
\end{equation}
\end{lemma}
\begin{proof}
Since the upper boundary $\Gamma_1^\varepsilon$ is not flat, one needs to take into account the variations of the normal direction ${\bf n}$ in order to estimate the $L^2$-norm of the trace of a function $\varphi\in H^1(\Omega^\varepsilon)$. Intregrating on vertical lines, we obtain
$$\begin{array}{rl}
\displaystyle \int_{\Gamma_1^\varepsilon}|\varphi|^2\,d\sigma=& \displaystyle \int_{\omega}|\varphi(x_1,h_\varepsilon(x_1))|^2\sqrt{1+\left({\eta_\varepsilon\over\varepsilon}\right)^2\left|h'\left({x_1\over \varepsilon}\right)\right|^2}\,dx_1\\
\leq &\displaystyle   C\left(1+{\eta_\varepsilon^2  \varepsilon^{-2}}\right)^{1\over 2}\int_\omega|\varphi(x_1,h_\varepsilon(x_1))|^2dx_1\\
\leq &\displaystyle C\left(1+{\eta_\varepsilon^2  \varepsilon^{-2}}\right)^{1\over 2}\int_{\omega}\left|\int_0^{h_\varepsilon(x_1)}\partial_{x_2}\varphi(x_1,x_2)dx_2\right|^2dx_1\\
\leq &\displaystyle C\left(1+{\eta_\varepsilon^2  \varepsilon^{-2}}\right)^{1\over 2}\eta_\varepsilon\int_{\Omega^\varepsilon}|\partial_{x_2} \varphi(x_1,x_2)|^2dx_1dx_2.
\end{array}
$$
Then, since $\eta_\varepsilon\ll \varepsilon$ or $\eta_\varepsilon\approx \varepsilon$, it holds
$$
\displaystyle \int_{\Gamma_1^\varepsilon}|\varphi|^2\,d\sigma\leq C \eta_\varepsilon\int_{\Omega^\varepsilon}|\partial_{x_2} \varphi(x_1,x_2)|^2dx_1dx_2,$$
which implies (\ref{trace_estim1}).\\

\noindent For the rescaled function, proceeding analogously, we get
$$\begin{array}{rl}
\displaystyle \int_{\widetilde \Gamma_1^\varepsilon}|\tilde \varphi|^2\,d\sigma=& \displaystyle \int_{\omega}|\tilde \varphi(x_1,h(x_1/\ep))|^2\sqrt{1+\left({1\over\varepsilon}\right)^2\left|h'\left({x_1\over \varepsilon}\right)\right|^2}\,dx_1\\
\leq &\displaystyle   C\left(1+{\varepsilon^{-2}}\right)^{1\over 2}\int_\omega|\tilde \varphi(x_1,h(x_1/\ep))|^2dx_1\\
\leq &\displaystyle C\left(1+{\varepsilon^{-2}}\right)^{1\over 2}\int_{\omega}\left|\int_0^{h(x_1/\ep)}\partial_{x_2}\tilde \varphi(x_1,x_2)dx_2\right|^2dx_1\\
\leq &\displaystyle C\left(1+{\varepsilon^{-2}}\right)^{1\over 2}\eta_\ep^2\int_{\widetilde \Omega^\varepsilon}|\eta_\ep^{-1}\partial_{x_2}\tilde  \varphi(x_1,x_2)|^2dx_1dx_2.
\end{array}
$$
Then,  we have that
$$ \int_{\widetilde \Gamma_1^\varepsilon}|\tilde \varphi|^2\,d\sigma\leq C\eta_\ep^2\ep^{-1}\int_{\widetilde \Omega^\varepsilon}|\eta_\ep^{-1}\partial_{x_2} \tilde \varphi(x_1,x_2)|^2dx_1dx_2,$$
which implies (\ref{trace_estim1_dil}).

\end{proof}

\begin{corollary}\label{G_lemma} For  $G^\varepsilon(x_1)=G(x_1/\varepsilon)$ with $G\in L^2(Z')$ and $Z'$-periodic function, we have the following estimate
\begin{equation}\label{Gestim1}\|G^\varepsilon\|_{L^2(\Gamma_1^\varepsilon)}\leq C\,.
\end{equation}
\end{corollary}
\begin{proof} The proof follows the line of estimates (\ref{trace_estim1}) and (\ref{trace_estim1}), just taking into account that $G\in L^2(Z')$.

\end{proof}
\begin{lemma}[A priori estimates]\label{lemma_estimates} Let $({\bf u}^\varepsilon, w^\varepsilon, T^\varepsilon)$ be the solution of the problem (\ref{system_1})-(\ref{BCTopBot1}). Then there hold the following estimates
\begin{eqnarray}
\displaystyle
\|{\bf u}^\varepsilon\|_{L^2(\Omega^\varepsilon)^2}\leq C\eta_\varepsilon^{1\over 2}, &\displaystyle
\|\nabla {\bf u}^\varepsilon\|_{L^2(\Omega^\varepsilon)^{2\times 2}}\leq C\eta_\varepsilon^{-{1\over 2}},&\medskip \label{estim_sol1}\\
\displaystyle
\|w^\varepsilon\|_{L^2(\Omega^\varepsilon)}\leq C\eta_\varepsilon^{1\over 2}, &\displaystyle
\|\nabla w^\varepsilon\|_{L^2(\Omega^\varepsilon)^{2}}\leq C\eta_\varepsilon^{-{1\over 2}},&\medskip\label{estim_sol2}\\
 \|T^\varepsilon\|_{L^2(\Omega^\varepsilon)}\leq C\eta_\varepsilon^{5\over 2}, &   \displaystyle\|\nabla T^\varepsilon\|_{L^2(\Omega^\varepsilon)^{2}}\leq C\eta_\varepsilon^{3\over 2}.\label{estim_sol3}&
\end{eqnarray}
Moreover, from the change of variables (\ref{dilatacion}),   there hold the following estimates for the rescaled unknowns
\begin{eqnarray}
\displaystyle
\| \tilde{\bf u}^\varepsilon\|_{L^2(\widetilde \Omega^\varepsilon)^2}\leq C, &\displaystyle
\|\nabla_{\eta_\varepsilon} \tilde{\bf u}^\varepsilon\|_{L^2(\widetilde\Omega^\varepsilon)^{2\times 2}}\leq C\eta_\varepsilon^{-1},& \label{estim_sol_dil1}\medskip
\\
\displaystyle
\|\tilde w^\varepsilon\|_{L^2(\widetilde\Omega^\varepsilon)}\leq C, &\displaystyle
\|\nabla_{\eta_\varepsilon} \tilde w^\varepsilon\|_{L^2(\widetilde\Omega^\varepsilon)^{2}}\leq C\eta_\varepsilon^{-1},\label{estim_sol_dil2}\medskip\\
 \displaystyle \|\tilde T^\varepsilon\|_{L^2(\widetilde\Omega^\varepsilon)}\leq C\eta_\varepsilon^{2}, &    \displaystyle\|\nabla_{\eta_\varepsilon} \tilde T^\varepsilon\|_{L^2(\widetilde\Omega^\varepsilon)^{2}}\leq C\eta_\varepsilon.& \label{estim_sol_dil3}
\end{eqnarray}
\end{lemma}
\begin{proof} We divide the proof in four steps. \\

\noindent {\it Step 1}. First, we multiply (\ref{system_1})$_3$ by $w^\ep$, integrate over  $\Omega^\ep$ to obtain
\begin{equation}\label{estim_proof}
\begin{array}{l}
\displaystyle L\int_{\Omega^\varepsilon}|\nabla w^\varepsilon|^2\,dx + {4N\over 1-N}\int_{\Omega^\varepsilon}|w^\varepsilon|^2\,dx\\
\noame
\displaystyle \qquad =-{M\over Pr}\int_{\Omega^\varepsilon}({\bf u}^\varepsilon \cdot \nabla w^\varepsilon)w^\varepsilon\,dx+{2N\over 1-N}\int_{\Omega^\varepsilon} {\rm rot}({\bf u}^\varepsilon) w^\varepsilon\,dx+{1\over\eta_\varepsilon^2}\int_{\Omega^\varepsilon}g\,w^\varepsilon\,dx.
\end{array}
\end{equation}
For the first term on the right-hand side, since ${\rm div}({\bf u}^\varepsilon)=0$ and $w^\varepsilon=0$ on $\partial\Omega^\varepsilon$, we get
\begin{equation}\label{estim_proof1}\int_{\Omega^\varepsilon}({\bf u}^\varepsilon \cdot \nabla w^\varepsilon)w^\varepsilon\,dx={1\over 2}\int_{\Omega^\varepsilon}{\bf u}^\varepsilon \cdot \nabla |w^\varepsilon|^2\,dx=-{1\over 2}\int_{\Omega^\varepsilon}|w^\varepsilon|^2{\rm div}({\bf u}^\varepsilon)\,dx=0.
\end{equation}
For the rest of the terms of the right-hand side, using the Cauchy-Schwarz inequality and the Poincar\'e inequality   (\ref{Poincare}), we get
\begin{equation}\label{estim_proof2}
\begin{array}{rl}
\displaystyle \left|\int_{\Omega^\varepsilon} {\rm rot}({\bf u}^\varepsilon) w^\varepsilon\,dx\right|&\leq \|\nabla {\bf u}^\varepsilon\|_{L^2(\Omega^\varepsilon)^{2\times 2}}\|w^\varepsilon\|_{L^2(\Omega^\varepsilon)}\leq C\eta_\varepsilon \|\nabla {\bf u}^\varepsilon\|_{L^2(\Omega^\varepsilon)^{2\times 2}}\|\nabla w^\varepsilon\|_{L^2(\Omega^\varepsilon)^{2 }},\\
\\
\displaystyle \left|{1\over \eta_\varepsilon^2}\int_{\Omega^\varepsilon} g\, w^\varepsilon\,dx_1dx_2\right|&\leq   \eta_\varepsilon^{-2}\|g\|_{L^2(\Omega^\varepsilon)}\|w^\varepsilon\|_{L^2(\Omega^\varepsilon)}\leq C\eta_\varepsilon^{-{1\over 2}}\|\nabla w^\varepsilon\|_{L^2(\Omega^\varepsilon)^{2 }}.
\end{array}
\end{equation}
Then,  taking into account (\ref{estim_proof1}) and (\ref{estim_proof2}) in (\ref{estim_proof}), we get
\begin{equation}\label{estim_proof3}
\|\nabla w^\varepsilon\|_{L^2(\Omega^\varepsilon)}\leq C\left(\eta_\varepsilon \|\nabla{\bf u}^\varepsilon\|_{L^2(\Omega^\varepsilon)}+\eta_\varepsilon^{-{1\over 2}}\right).
\end{equation}
{\it Step 2}. We multiply (\ref{system_1})$_4$ by $T^\ep$, integrate over  $\Omega^\ep$ to obtain
\begin{equation}\label{estim_proof4}
\begin{array}{l}
\displaystyle \int_{\Omega^\varepsilon}|\nabla T^\varepsilon|^2\,dx+\eta_\varepsilon k\int_{\Gamma^\varepsilon_1}|T^\varepsilon|^2\,d\sigma\\
\noame
\qquad\displaystyle
= - \int_{\Omega^\varepsilon}({\bf u}^\varepsilon\cdot \nabla)T^\varepsilon T^\varepsilon\,dx  +D\int_{\Omega^\varepsilon}(\nabla^\perp w^\varepsilon\cdot \nabla)T^\varepsilon T^\varepsilon\,dx +\eta_ \varepsilon k \int_{\Gamma^\varepsilon_1}G^\varepsilon T^\varepsilon\,d\sigma
.\end{array}
\end{equation}
For the first term on the right-hand side of (\ref{estim_proof4}), since ${\bf u}^\varepsilon=0$ on $\partial\Omega^\varepsilon$, ${\rm div}({\bf u}^\varepsilon)=0$ in $\Omega^\varepsilon$ and $T^\varepsilon=0$ on $\Sigma_i^\varepsilon$, $i=-1/2, 1/2$, we get
\begin{equation}\label{estim_proof5}
\int_{\Omega^\varepsilon}({\bf u}^\varepsilon\cdot \nabla)T^\varepsilon T^\varepsilon\,dx={1\over 2}\int_{\Omega^\varepsilon}{\bf u}^\varepsilon\cdot \nabla |T^\varepsilon|^2dx=-{1\over 2}\int_{\Omega^\varepsilon}|T^\varepsilon|^2{\rm div}{\bf u}^\varepsilon\,dx=0.
\end{equation}
For the  second term on the right-hand side of (\ref{estim_proof4}), we have
\begin{equation}\label{estim_proof6}
\begin{array}{l}
\displaystyle \int_{\Omega^\varepsilon}\nabla^\perp w^\varepsilon\cdot \nabla T^\varepsilon T^\varepsilon\,dx  \displaystyle =
{1\over 2}\int_{\Omega^\varepsilon}\nabla^\perp w^\varepsilon \cdot \nabla(T^\varepsilon)^2dx =-{1\over 2}\int_{\Omega^\varepsilon}\nabla w^\varepsilon \times (\nabla (T^\varepsilon)^2)\,dx \\
\\
\qquad \displaystyle ={1\over 2}\int_{\Omega^\varepsilon} w^\varepsilon {\rm rot}(\nabla (T^\varepsilon)^2)dx -{1\over 2}\int_{\Omega^\varepsilon} {\rm rot}(w^\varepsilon \nabla (T^\varepsilon)^2)dx \\
\\
\qquad\displaystyle =-{1\over 2}\int_{\Omega^\varepsilon}{\bf n}\times (w^\varepsilon \nabla (T^\varepsilon)^2)dx =-{1\over 2}\int_{\Omega^\varepsilon}w^\varepsilon\left({\partial T^\varepsilon\over \partial x_2} T^\varepsilon n_1-{\partial T^\varepsilon\over \partial x_1} T^\varepsilon n_2\right)dx =0,
\end{array}
\end{equation}
where we have used that $w^\varepsilon=0$ on $\partial\Omega^\varepsilon$, the identity
$${\rm rot}(\nabla (T^\varepsilon)^2)={\rm rot}\left(2{\partial T^\varepsilon\over \partial x_1} T^\varepsilon, 2{\partial T^\varepsilon\over \partial x_2}T^\varepsilon\right)=2{\partial^2 T^\varepsilon\over \partial x_2\partial x_1}T^\varepsilon+2{\partial T^\varepsilon\over \partial x_2}{\partial T^\varepsilon\over \partial x_1}-2{\partial^2 T^\varepsilon\over \partial x_1\partial x_2}T^\varepsilon-2{\partial T^\varepsilon\over \partial x_1}{\partial T^\varepsilon\over \partial x_2}=0.
$$
and
$$\int_{\Omega^\varepsilon}{\partial w^\varepsilon\over \partial x_2}T^\varepsilon\,dx_1dx_2=-\int_{\Omega^\varepsilon}w^\varepsilon{\partial T^\varepsilon\over \partial x_2}dx_1dx_2.$$

\noindent It remains to estimate the third term of the right-hand side of   (\ref{estim_proof4}). To do this, from Caychy-Schwarz's inequality, Lemma \ref{trace_lemma} applied  to $T^\varepsilon$ and Corollary \ref{G_lemma}, we get
\begin{equation}\label{estim_proof7}\left|\eta_\varepsilon k \int_{\Gamma_1^\varepsilon}G^\varepsilon T^\varepsilon\,d\sigma\right|\leq \eta_\varepsilon\|G^\varepsilon \|_{L^2(\Gamma_1^\varepsilon)}\|T^\varepsilon\|_{L^2(\Gamma_1^\varepsilon)}\leq C  \eta_\ep^{3\over 2}\|\nabla T^\varepsilon\|_{L^2(\Omega^\varepsilon)^2}.
\end{equation}
Then,  taking into account (\ref{estim_proof5}) - (\ref{estim_proof7}) in (\ref{estim_proof4}), we have
\begin{equation}\label{estim_proof8}
\|\nabla T^\varepsilon\|_{L^2(\Omega^\varepsilon)}\leq C\eta_\varepsilon^{3\over 2},
\end{equation}
which is the second estimate in (\ref{estim_sol3}). From the Poincar\'e inequality (\ref{Poincare}), we   get the first estimate in (\ref{estim_sol3}).
\\

\noindent {\it Step 3}. We multiply (\ref{system_1})$_1$ by ${\bf u}^\ep$, integrate over  $\Omega^\ep$ to obtain
\begin{equation}\label{estim_proof9}
 \begin{array}{rl}
\displaystyle {1\over 1-N}\int_{\Omega_\ep}|\nabla{\bf u}^\varepsilon|^2\,dx = &\displaystyle -{1\over Pr}\int_{\Omega^\varepsilon}({\bf u}^\varepsilon\cdot \nabla){\bf u}^\varepsilon{\bf u}^\varepsilon\,dx  +{2N\over 1-N}\int_{\Omega^\varepsilon}\nabla^\perp w^\varepsilon \cdot {\bf u}^\varepsilon \,dx
\\
\noame &
\displaystyle
+Ra\int_{\Omega^\varepsilon}T^\varepsilon ({\bf e}_2\cdot {\bf u}^\varepsilon)\,dx+{1\over \eta_\varepsilon^2}\int_{\Omega^\varepsilon}f_1 ({\bf e}_1\cdot {\bf u}^\varepsilon)\,dx,
\\
\noame &
\displaystyle
 +{1\over Pr}{1\over \eta_\varepsilon^2}q_{-1/2}\int_{\Sigma_{-1/2}^\varepsilon}\varphi\cdot {\bf e_1}\,dx_2-{1\over Pr}{1\over \eta_\varepsilon^2}q_{1/2}\int_{\Sigma_{1/2}^\varepsilon} \varphi\cdot {\bf e_1}\,dx_2.
\end{array}
\end{equation}
The first term on the right-han side of (\ref{estim_proof9}) satisfies
\begin{equation}\label{estim_proof09}
\int_{\Omega^\varepsilon}({\bf u}^\varepsilon\cdot \nabla){\bf u}^\varepsilon{\bf u}^\varepsilon\,dx_1dx_2={1\over 2}\int_{\Omega^\varepsilon}{\bf u}^\varepsilon \cdot \nabla |{\bf u}^\varepsilon|^2dx_1dx_2=-{1\over 2}\int_{\Omega^\varepsilon}|{\bf u}^\varepsilon|^2{\rm div}\,{\bf u}^\varepsilon\,dx_1dx_2=0.
\end{equation}
We estimate the rest of the terms on the right-hand side of (\ref{estim_proof9}) by using the Poincar\'e inequality (\ref{Poincare}) and using (\ref{estim_proof3}), we get
\begin{equation}\label{estim_proof10}
\begin{array}{rl}
\displaystyle \left|\int_{\Omega^\varepsilon}\nabla^\perp w^\varepsilon{\bf u}^\varepsilon \,dx \right|&\leq \|\nabla w^\varepsilon\|_{L^2(\Omega^\varepsilon)}\|{\bf u}^\varepsilon\|_{L^2(\Omega^\varepsilon)}\leq C\eta_\varepsilon\|\nabla w^\varepsilon\|_{L^2(\Omega^\varepsilon)^{2}}\|\nabla {\bf u}^\varepsilon\|_{L^2(\Omega^\varepsilon)^{2\times 2}}\\
&\leq C\eta_\varepsilon^{2}\|\nabla {\bf u}^\varepsilon\|^2_{L^2(\Omega^\varepsilon)^{2\times 2}}+\eta^{1\over 2}_\varepsilon\|\nabla {\bf u}^\varepsilon\|_{L^2(\Omega^\varepsilon)^{2\times 2}},\\
\\
 \displaystyle \left|\int_{\Omega^\varepsilon}T^\varepsilon({\bf e}_2\cdot {\bf u}^\varepsilon)\,dx \right|&\leq \|T^\varepsilon\|_{L^2(\Omega^\varepsilon)}\|{\bf u}^\ep\|_{L^2(\Omega^\varepsilon)^2}\leq C\eta_\varepsilon \|T^\varepsilon\|_{L^2(\Omega^\varepsilon)}\|\nabla {\bf u}^\ep\|_{L^2(\Omega^\varepsilon)^{2\times 2}}
\\
& \leq C\eta_\varepsilon^{7\over 2}\|\nabla {\bf u}^\ep\|_{L^2(\Omega^\varepsilon)^{2\times 2}},
\\
\\
\displaystyle \left|\eta_\ep^{-2}\int_{\Omega^\varepsilon}f_1{\bf u}^\varepsilon\,dx \right|&\leq  \eta_\varepsilon^{-2}\|f_1\|_{L^2(\Omega^\varepsilon)}\|{\bf u}^\varepsilon\|_{L^2(\Omega^\varepsilon)^2}\leq C\eta_\varepsilon^{-{1\over 2}}\|\nabla{\bf u}^\varepsilon\|_{L^2(\Omega^\varepsilon)^{2\times 2}},
\end{array}
\end{equation}
and
\begin{equation}\label{aaa}
\begin{array}{rl}
\displaystyle {1\over \eta_\varepsilon^2}{1\over Pr}\left|q_{-1/2}\int_{\Sigma_{-1/2}^\varepsilon}{\bf u}^\ep\cdot {\bf e_1}\,dx_2-q_{1/2}\int_{\Sigma_{1/2}^\varepsilon} {\bf u}^\ep\cdot {\bf e_1}\,dx_2\right|
&\displaystyle =|{\rm div}((q_{-1/2}+(q_{1/2}-q_{-1/2})(x+1/2)){\bf u}^\varepsilon|\\
\\
&\displaystyle\leq C \eta_\varepsilon^{-2}\|1\|_{L^2(\Omega^\varepsilon)}\|{\bf u}^\varepsilon\|_{L^2(\Omega^\varepsilon)^2}\leq C\varepsilon^{-{1\over 2}}\|\nabla {\bf u}^\varepsilon\|_{L^2(\Omega^\varepsilon)^{2\times 2}}.
\end{array}
\end{equation}
Taking into account  (\ref{estim_proof09}) and estimates (\ref{estim_proof10}) in (\ref{estim_proof9}), we get
$$\|\nabla{\bf u}^\varepsilon\|_{L^2(\Omega^\varepsilon)^{2\times 2}}\leq C\eta_\varepsilon^{-{1\over 2}},$$
which is the second estimate (\ref{estim_sol1}). By using the Poincar\'e inequality, we get the first estimate in  (\ref{estim_sol1}). For the microrotation, from (\ref{estim_proof3}), we get the estimates of the microrotation (\ref{estim_sol2}).\\

\noindent {\it Step 4.} Finally, the estimates of the rescaled unknown  are obtained by applying to estimates of the unknowns the change of variables (\ref{dilatacion}).

\end{proof}

\subsection{The extension of $(\tilde {\bf u}^\varepsilon, \tilde w^\varepsilon, \tilde T^\varepsilon)$  to the whole domain}
The sequence of solutions $(\tilde{\bf u}^\ep, \tilde w^\ep, \tilde T^\varepsilon)$ is defined in a varying set $\widetilde\Omega^\ep$, but not defined in a fixed domain independent of $\ep$. In order to pass to the limit if $\ep$ tends to zero, convergences in fixed Sobolev spaces (defined in $\Omega$) are used, which requires first that $(\tilde{\bf  u}^\ep, \tilde w^\ep, \tilde T^\varepsilon)$ be extended to the whole domain $\Omega$. We extend each unknown in the following:

\begin{itemize}
\item[--] From the boundary conditions satisfied by $\tilde{\bf u}^\varepsilon$ and $\tilde w^{\varepsilon}$, we extend them  by zero in $\Omega\setminus \widetilde{\Omega}^{\varepsilon}$ and denote the extensions by  $ \tilde{\bf U}^\varepsilon$ and $\tilde W^{\varepsilon}$, respectively.

\item[--]   For the temperature $\tilde T^\varepsilon$, we use the extension operator described in \cite[Lemma 2.3]{Pereira_p_laplace} called $\mathcal{P}^\varepsilon$ which allows us to extend functions from $H^1(\widetilde\Omega^\varepsilon)$, which are zero on the lateral boundaries, to $H^1(\Omega)$. Moreover,  this extension satisfies
\begin{equation}\label{Extension_op}\begin{array}{l}
\|\mathcal{P}^\varepsilon(\tilde \varphi)\|_{L^2(\Omega)}\leq C\|\tilde \varphi\|_{L^2(\widetilde\Omega^\varepsilon)},\\
\noame
\displaystyle \|\partial_{x_1} \mathcal{P}^\varepsilon (\tilde \varphi)\|_{L^2(\Omega)}\leq C\left(\|\partial_{x_1} \tilde \varphi\|_{L^2(\widetilde \Omega^\varepsilon)}+{1\over \varepsilon} \|\partial_{z_2}\tilde \varphi\|_{L^2(\widetilde \Omega^\varepsilon)}\right),\\
\noame
\|\partial_{z_2} \mathcal{P}^\varepsilon (\tilde \varphi)\|_{L^2(\Omega)}\leq C \|\partial_{z_2}\tilde \varphi\|_{L^2(\widetilde \Omega^\varepsilon)},
\end{array}
\end{equation}
for every function $\tilde \varphi\in H^1(\widetilde\Omega_\varepsilon)$. Thus, we denote by $\tilde \theta^\varepsilon$ the extension of $\tilde T^\varepsilon$, i.e. $\tilde \theta^\varepsilon=\mathcal{P}^\varepsilon (\tilde T^\varepsilon)$. \\
\end{itemize}
We have the following result.
\begin{lemma}[Estimates of extended functions]\label{lemma_estimates2} The extended functions $(\tilde U^\ep, \tilde W^\ep, \tilde \theta^\varepsilon)$ of $(\tilde u^\ep, \tilde w^\ep, \tilde T^\varepsilon)$ satisfy the following estimates
\begin{eqnarray}
\displaystyle
&&\|\tilde{\bf U}^\varepsilon\|_{L^2(  \Omega )^2}\leq C , \quad \displaystyle
\|\nabla_{\eta_\varepsilon}  \tilde{\bf U}^\varepsilon\|_{L^2( \Omega )^{2\times 2}}\leq C\eta_\varepsilon^{-1}, \label{estim_sol_dil_ext_u}
\medskip
\\
\displaystyle
&&\|\tilde W^\varepsilon\|_{L^2( \Omega)}\leq C, \quad \displaystyle
\|\nabla_{\eta_\varepsilon} \tilde W^\varepsilon\|_{L^2( \Omega )^{2}}\leq C\eta_\varepsilon^{-1}, \label{estim_sol_dil_ext_w}
\medskip
\\
&& \|\tilde \theta^\varepsilon \|_{L^2 (\Omega )}\leq C\eta_\varepsilon^{2}, \quad     \displaystyle\|\nabla_{\eta_\varepsilon} \tilde \theta^\varepsilon \|_{L^2( \Omega )^{2}}\leq C\eta_\varepsilon. \label{estim_sol_dil_ext_T}
\end{eqnarray}

\end{lemma}
\begin{proof}  Estimates for the extension of $\tilde{\bf U}^\varepsilon$ and $\tilde W^\varepsilon$ are obtained straightforward from  (\ref{estim_sol_dil1}) and (\ref{estim_sol_dil2}), respectively.  For the extension of the temperature $\tilde \theta^\varepsilon$, from (\ref{estim_sol_dil3})  and (\ref{Extension_op}), we deduce (\ref{estim_sol_dil_ext_T}).
\\
\end{proof}
\subsection{Estimates for pressure}
Let us first give a more accurate estimate for pressure $p^\varepsilon$. For this, we need to recall a version of the decomposition result for $p^\varepsilon$ whose proof can be found in \cite[Corollary 3.4]{CLS2}  (see also \cite{Bonn_Paz_SG, CLS1}).
\begin{proposition}\label{prop_decomposition}
The following decomposition for $p^\varepsilon\in L^2_0(\Omega^\varepsilon)$ holds
\begin{equation}\label{decompositionp}
p^\varepsilon=p^\varepsilon_0+p^\varepsilon_1,
\end{equation}
where $p^\varepsilon_0\in H^1(\omega)$, which is independent of $x_2$, and $p^\varepsilon_1\in L^2(\Omega^\varepsilon)$. Moreover, the following estimates hold
$$\eta_\ep^{3\over 2}\|p^\varepsilon_0\|_{H^1(\omega)}+\|p^\varepsilon_1\|_{L^2(\Omega^\varepsilon)}\leq C\|\nabla p^\varepsilon\|_{H^{-1}(\Omega^\varepsilon)^2},$$
that is
\begin{equation}\label{decomposition_estimates}
\|p^\varepsilon_0\|_{H^1(\omega)}\leq C\eta_\varepsilon^{-{3\over 2}}\|\nabla p^\varepsilon\|_{H^{-1}(\Omega^\varepsilon)^2},\quad
\|p^\varepsilon_1\|_{L^2(\Omega^\varepsilon)}\leq C\|\nabla p^\varepsilon\|_{H^{-1}(\Omega^\varepsilon)^2}.
\end{equation}
\end{proposition}
\noindent We denote by $\tilde p^\varepsilon_1$ the rescaled function associated with $p^\varepsilon_1$ defined by $\tilde p^\ep_1(x_1,z_2)=(x_1,\eta_\ep z_2)$ for a.e. $(x_1,z_2)\in \widetilde\Omega^\ep$. As consequence, we have the following result.
\begin{corollary}\label{Lem_estim_pressures} The pressures $p^\varepsilon_0$,  $p^\varepsilon_1$ and $\tilde p^\ep_1$ satisfy the following estimates
\begin{equation}\label{estim_p0p1}
\begin{array}{c}
\displaystyle
\|p^\varepsilon_0\|_{H^1(\omega)}\leq C\eta_\varepsilon^{-2},\\
\\
\displaystyle   \|p^\varepsilon_1\|_{L^2(\Omega^\varepsilon)}\leq C\eta_\ep^{-{1\over 2}},\quad  \|\tilde p^\varepsilon_1\|_{L^2(\widetilde \Omega^\varepsilon)}\leq C\eta_\ep^{-1}.
\end{array}
\end{equation}
\end{corollary}
\begin{proof} Thank to (\ref{decomposition_estimates}), we just need to obtain the estimate for $\nabla p^\ep$ given by
\begin{equation}\label{estim_nabla_p}
\|\nabla  p^\ep\|_{H^{-1}(\Omega^\varepsilon)^2}\leq C\eta_\ep^{-{1\over 2}},
\end{equation}
to derive (\ref{estim_p0p1}).
To do this, we consider $\varphi\in H^1_0(\Omega^\varepsilon)$, and taking into account the variational formulation (\ref{form_var_u}), we get
\begin{equation}\label{equality_duality_0}
\begin{array}{rl}
\displaystyle
\left\langle \nabla p^\varepsilon,\varphi\right\rangle=&\displaystyle
- {Pr\over 1-N}\int_{\Omega_\varepsilon}\nabla {\bf u}_\ep : \nabla \varphi\,dx -\int_{\Omega^\varepsilon} ({\bf u}^\varepsilon\cdot \nabla){\bf u}^\varepsilon \varphi\,dx\\
\noame &\displaystyle
+{2N\,Pr\over 1-N}\int_{\Omega^\varepsilon}\nabla^\perp w^\varepsilon \cdot  \varphi\,dx+Pr\,Ra\int_{\Omega^\varepsilon}  T^\varepsilon({\bf e}_2\cdot  \varphi)\,dx\\
\noame &\displaystyle
+{Pr\over \eta_\varepsilon^2}\int_{\Omega^\varepsilon}f_1 ({\bf e}_1\cdot  \varphi)\,dx.
\end{array}\end{equation}
Estimating the terms on the right-hand side of  (\ref{equality_duality_0}) using Lemmas \ref{Poincare_lemma} and \ref{lemma_estimates}, we get
$$\begin{array}{rcl}
\displaystyle\left|{Pr\over 1-N}\int_{\Omega_\varepsilon}\nabla {\bf    u}_\ep : \nabla \varphi\,dx \right| &\leq &\displaystyle
C\|\nabla{\bf  u}^\varepsilon\|_{L^2( \Omega^\varepsilon)^{2\times 2}}\|\nabla\varphi\|_{L^2(\Omega^\varepsilon)^{2\times 2}}\leq C\eta_\varepsilon^{-{1\over 2}}\|  \varphi\|_{H^1_0( \Omega^\varepsilon)^2},\\
\\
\displaystyle
 \left|\int_{ \Omega^\varepsilon} ({\bf  u}^\varepsilon\cdot \nabla){\bf   u}^\varepsilon  \varphi\,dx\right| &\leq   & \displaystyle\|{\bf  u}^\varepsilon\|_{L^4(\Omega^\varepsilon)^2}\|\nabla {\bf   u}^\varepsilon\|_{L^2(\Omega^\varepsilon)^{2\times 2}}\|\varphi\|_{L^4( \Omega^\varepsilon)^2}\\
\noame
  &\leq  &\displaystyle  C\eta_\ep\|\nabla{\bf  u}^\varepsilon\|^2_{L^2(\Omega^\varepsilon)^{2\times 2}}\|\nabla\varphi\|_{L^2(\Omega^\varepsilon)^{2\times 2}}\leq \| \varphi\|_{H^1_0( \Omega^\varepsilon)^2},
\\
\\
\displaystyle\left|{2N\,Pr\over 1-N}\int_{\Omega^\varepsilon}\nabla^\perp   w^\varepsilon  \varphi\,dx
\right|&\leq &\displaystyle C  \|\nabla  w^\varepsilon\|_{L^2(\Omega^\varepsilon)^2}\|\varphi\|_{L^2(\Omega^\varepsilon)^2}\leq \eta_\varepsilon^{1\over 2}\|\varphi\|_{H^1_0(\Omega^\varepsilon)^2},
\\
\\
\displaystyle\left|Pr\,Ra\int_{ \Omega^\varepsilon}  T^\varepsilon({\bf e}_2\cdot \varphi)\,dx\right|
&\leq &\displaystyle C\| T^\varepsilon\|_{L^2( \Omega^\varepsilon)}\|\varphi\|_{L^2(\Omega^\varepsilon)^2}\leq C\eta_\varepsilon^{7\over 2}\| \varphi\|_{H^1_0(\Omega^\varepsilon)^2},\\
\\
\displaystyle\left|{Pr\over \eta_\varepsilon^2}\int_{\Omega^\varepsilon}f_1 ({\bf e}_1\cdot \varphi)\,dx\right|& \leq &\displaystyle C\eta_\varepsilon^{-2}\|f_1\|_{L^2(\Omega^\varepsilon)}\| \varphi\|_{L^2(\Omega^\varepsilon)^2}\leq
C\eta_\varepsilon^{-{1\over 2}}\| \varphi\|_{H^1_0(\Omega^\varepsilon)^2},
\end{array}$$
which together with (\ref{equality_duality_0}) gives
$$\left\langle \nabla p^\varepsilon,\varphi\right\rangle\leq C\eta_\varepsilon^{-{1\over 2}}\| \varphi\|_{H^1_0(\Omega^\varepsilon)^2},\quad\forall\,\varphi \in H^1_0(\Omega^\varepsilon)^2.$$
This gives the desired estimate  (\ref{estim_nabla_p}), which finishes the proof.

\end{proof}

\subsection{Adaptation of the unfolding method} The change of variables (\ref{dilatacion}) does not provide the information we need about the behavior of rescaled unknown in the microstructure associated to $\widetilde\Omega_\ep$. To solve this difficulty, we use an adaptation of the unfolding method (see \cite{Ciora},\cite{Ciora2} for more details) introduced to this context in \cite{Anguiano_SG} (see also related methods for different domains with roughness \cite{CLS0}, \cite{Pazanin_SG2}, \cite{grau1}, \cite{grau2}, \cite{grauMicRough} and for thin porous media  \cite{AnguianoSG_coupled}, \cite{AnguianoSG_ZAMP}, \cite{Anguiano_SG_Network}, \cite{grauMicTPM}).

\noindent Let us recall that this adaptation of the unfolding method divides the domain $\widetilde\Omega_\ep$ in cubes of lateral length $\ep$ and vertical length $h(z_1)$. Thus, given the unknowns $ \tilde{{\bf u}}^{\varepsilon}, \tilde w^\ep,  \tilde T^\varepsilon$, $p^\ep_0$ and $\tilde p^\ep_1$, we define
\begin{eqnarray}\label{uhat}
\hat{{\bf u}}^{\varepsilon}(x_1,z)=\tilde{{\bf u}}^{\varepsilon}\left( {\varepsilon}\kappa\left(\frac{x_1 }{{\varepsilon}} \right)+{\varepsilon}z_1 ,z_2 \right)\text{\ \ a.e. \ }(x_1,z)\in \omega\times Z,
\\
\noame
\hat{w}^{\varepsilon}(x_1,z)=\tilde{w}^{\varepsilon}\left( {\varepsilon}\kappa\left(\frac{x_1 }{{\varepsilon}} \right)+{\varepsilon}z_1 ,z_2 \right)\text{\ \ a.e. \ }(x, z)\in \omega\times Z,\label{what}
\\
\noame
\hat{T}^{\varepsilon}(x_1,z)=\tilde{T}^{\varepsilon}\left( {\varepsilon}\kappa\left(\frac{x_1 }{{\varepsilon}} \right)+{\varepsilon}z_1,z_2 \right)\text{\ \ a.e. \ }(x_1,z)\in \omega\times Z,\label{That}
\\
\noame
\hat{p}^{\varepsilon}_0(x_1,z_1)= {p}^{\varepsilon}_0\left( {\varepsilon}\kappa\left(\frac{x_1}{{\varepsilon}} \right)+{\varepsilon}z_1\right)\text{\ \ a.e. \ }(x_1,z_1)\in \omega\times Z'.\label{P0hat}\\
\noame
\hat{p}^{\varepsilon}_1(x_1,z)=\tilde{p}^{\varepsilon}_1\left( {\varepsilon}\kappa\left(\frac{x_1}{{\varepsilon}} \right)+{\varepsilon}z_1,z_2 \right)\text{\ \ a.e. \ }(x_1,z)\in \omega\times Z.\label{P1hat}
\end{eqnarray}
where  the function $\kappa$ is defined by (\ref{kappa_fun}).\\

  \begin{remark} For $k'\in T_\ep$, the restriction of $(\hat {\bf u}^\ep, \hat w^\ep, \hat T^\ep, \hat p^\ep_1)$ to $Z'_{k',\ep}\times Z$ does not depend on $x_1$, while as a function of $z$ it is  obtained from $( \tilde{\bf  u}^\ep, \tilde w^\ep, \tilde \theta^\ep, \tilde p^\ep_1)$  by using the change of variables
  $$z_1={x_1-\ep k'\over \ep},$$
  which transform $Z_{k',\ep}$ into $Z$. Analogously, the restriction of $\hat p^\ep_0$ to $Z'_{k',\ep}\times Z'$ does not depend on $x_1$, while as a function of $z_1$ it is obtained from $p^\ep_0$ by using the previous change of variables.
  \end{remark}
\noindent We are now in position to obtain estimates for the unfolded unknowns $(\hat{{\bf u}}^{\varepsilon}, \hat w^\ep, \hat T^\varepsilon, \hat{p}^{\varepsilon}_0, \hat p^\ep_1)$.
 \begin{lemma}\label{estimates_hat}
 There exists a constant $C>0$ independent of $\ep$, such that $\hat {\bf u}^\ep$, $\hat w^\ep$, $\hat T^\varepsilon$, $\hat p^\ep_0$ and $\hat p^\ep_1$ defined by (\ref{uhat})--(\ref{P1hat}) respectively satisfy
 \begin{eqnarray}
 \|\hat{\bf  u}^\ep\|_{L^2(\omega\times Z)^2}\leq C,&
 \|\partial_{z_1} \hat {\bf u}^\ep\|_{L^2(\omega\times Z)^{2}}\leq C\varepsilon\eta_\ep^{-1},&
 \|\partial_{z_2} \hat {\bf u}^\ep\|_{L^2(\omega\times Z)^{2}}\leq C,\label{estim_u_hat}\\
 \noame
  \|\hat w^\ep\|_{L^2(\omega\times Z)}\leq C ,&
 \|\partial_{z_1}\hat w^\ep\|_{L^2(\omega\times Z)}\leq C\varepsilon\eta_\ep^{-1},&
 \|\partial_{z_2}\hat w^\ep\|_{L^2(\omega\times Z)}\leq C,\label{estim_w_hat}\\
 \noame
  \|\hat T^\ep\|_{L^2(\omega\times Z)}\leq C\eta_\ep^2,&
  \|\partial_{z_1}\hat T^\ep\|_{L^2(\omega\times Z)}\leq C\ep \eta_\varepsilon,&
  \|\partial_{z_2}\hat T^\ep\|_{L^2(\omega\times Z)}\leq C\eta_\varepsilon^2,\label{estim_T_hat}
  \\
  \noame
    \|\hat p^\ep_0\|_{L^2(\omega\times Z')}\leq C\eta_\ep^{-2} ,&
  \|\partial_{z_1}\hat p^\ep_0\|_{L^2(\omega\times Z')}\leq C\ep\eta_\ep^{-2},&  \|\hat p^\ep_1\|_{L^2(\omega\times Z)}\leq C\eta_\varepsilon^{-1}.
\label{estim_P01_hat}
  \end{eqnarray}
 \end{lemma}
\begin{proof}From the proof of Lemma 4.9 in \cite{Anguiano_SG} in the case $p=2$, we have the following properties concerning the estimates of a function $\tilde \varphi^\ep$ and its respective unfolding function $\hat \varphi^\ep$:
\begin{eqnarray}
&\|\hat \varphi^\ep\|_{L^2(\omega\times Z)^2}=\|\tilde \varphi^\ep\|_{L^2(\widetilde\Omega^\ep)^2}\,,&\nonumber\\
\noame
&
\|\partial_{z_1}\hat \varphi^\ep\|_{L^2(\omega\times Z)^{2\times 1}}=\ep\|\partial_{x_1}\tilde \varphi^\ep\|_{L^2(\widetilde\Omega^\ep)^{2}},\quad \|\partial_{z_2}\hat \varphi^\ep\|_{L^2(\omega\times Z)^{2}}=\|\partial_{z_2}\tilde \varphi^\ep\|_{L^2(\widetilde\Omega^\ep)^{2}}\,.&\label{relationpressure}
\end{eqnarray}
Thus, combining previous estimates of $\hat \varphi^\ep$ with estimates for $\tilde{\bf  u}^\ep$, $\tilde w^\ep$,  $\tilde T^\varepsilon$ given in Lemma \ref{lemma_estimates} and estimate for  $\tilde p^\ep_1$ given in Corollary \ref{Lem_estim_pressures}, we respectively get (\ref{estim_u_hat}),  (\ref{estim_w_hat}),  (\ref{estim_T_hat}) and (\ref{estim_P01_hat}).

\end{proof}

\subsection{Weak variational formulation} We give the equivalent weak variational formulation of system  (\ref{system_1})-(\ref{BCTopBot1}) and the rescaled system (\ref{system_1_dil})-(\ref{BCTop1_dil}), which will be useful in next sections in order to obtain the limit system taking into account the effects of the rough boundary.

\noindent Taking into account the decomposition of the pressure and
$$\begin{array}{rl}\displaystyle
\left\langle \nabla  p^\varepsilon, \varphi\right\rangle=&\displaystyle
\int_{\Omega^\ep}\partial_{x_1}p^\ep_0(x_1)\,  \varphi_1\,dx_1dz_2-\int_{\widetilde \Omega^\ep}\tilde p^\ep_1\,{\rm div} ( \varphi)\,dx_1dz_2,\quad \forall\, \varphi\in H^1_0(\Omega^\ep)^2,\\
\end{array}$$
 then, the weak variational formulation for  system (\ref{system_1})-(\ref{BCTopBot1}) is the following
\begin{equation}\label{form_var_u}
 \begin{array}{l}
\displaystyle {1\over 1-N}\int_{\Omega_\ep}\nabla{\bf u}^\varepsilon\cdot \nabla \varphi\,d{x}
+{1\over Pr}\int_{ \Omega^\ep}\partial_{x_1} p^\ep_0(x_1)\,\varphi_1\,dx_1dx_2-{1\over Pr}\int_{  \Omega^\ep}   p^\ep_1\,{\rm div}(\varphi)\,dx\\
\noame
\quad \displaystyle =-{1\over Pr}\int_{\Omega^\varepsilon}({\bf u}^\varepsilon\cdot \nabla){\bf u}^\varepsilon\varphi\,dx +{2N\over 1-N}\int_{\Omega^\varepsilon}\nabla^\perp w^\varepsilon\cdot \varphi \,dx

+Ra\int_{\Omega^\varepsilon}T^\varepsilon ({\bf e}_2\cdot \varphi)\,dx+{1\over \eta_\varepsilon^2}\int_{\Omega^\varepsilon}f_1 ({\bf e}_1\cdot \varphi)\,dx
\end{array}
\end{equation}

\begin{equation}\label{form_var_w}
 \begin{array}{l}
\displaystyle L\int_{\Omega^\varepsilon}\nabla w^\varepsilon\cdot \nabla \psi\,dx + {4N\over 1-N}\int_{\Omega^\varepsilon}w^\varepsilon \psi\,dx\\
\noame
\displaystyle \qquad\qquad\qquad\qquad =-{M\over Pr}\int_{\Omega^\varepsilon}({\bf u}^\varepsilon \cdot \nabla w^\varepsilon)\psi\,dx+{2N\over 1-N}\int_{\Omega^\varepsilon} {\rm rot}({\bf u}^\varepsilon) \psi\,dx+{1\over\eta_\varepsilon^2}\int_{\Omega^\varepsilon}g\, \psi\,dx,
\end{array}
\end{equation}

\begin{equation}\label{form_var_T}
 \begin{array}{l}
\displaystyle \int_{\Omega^\varepsilon}\nabla T^\varepsilon\cdot \nabla \phi\,dx +\eta_\varepsilon k\int_{\Gamma^\varepsilon_1}T^\varepsilon \phi\,d\sigma\\
\noame
\quad\qquad\qquad\qquad\qquad\displaystyle
= - \int_{\Omega^\varepsilon}({\bf u}^\varepsilon\cdot \nabla)T^\varepsilon\phi\,dx  +D\int_{\Omega^\varepsilon}\nabla^\perp w^\varepsilon\cdot \nabla T^\varepsilon\phi\,dx + \eta_\varepsilon k \int_{\Gamma^\varepsilon_1}G^\varepsilon\phi\,d\sigma\,.\end{array}
\end{equation}
for every $\varphi \in H^1_0(\Omega^\ep)^2$, $\psi\in H^1_0(\Omega^\ep)$ and $\phi\in H^1(\Omega^\ep)$  such that $\phi=0\ \hbox{on}\ \partial \Omega^\varepsilon\setminus \Gamma^\varepsilon_1$.\\

\noindent  The equivalent weak variational formulation for the rescaled system (\ref{system_1_dil})-(\ref{BCTop1_dil}) reads as follows
\begin{equation}\label{form_var_1_udil}
 \begin{array}{l}
\displaystyle {1\over 1-N}\int_{\widetilde \Omega_\ep}\nabla_{\eta_\varepsilon} \tilde {\bf u}^\varepsilon\cdot \nabla_{\eta_\varepsilon} \tilde \varphi\,dx_1dz_2
+{1\over Pr}\int_{\widetilde \Omega^\ep}\partial_{x_1} p^\ep_0(x_1)\,\tilde \varphi_1\,dx_1dz_2-{1\over Pr}\int_{\widetilde \Omega^\ep} \tilde p^\ep_1\,{\rm div}_{\eta_\ep}\,\tilde \varphi\,dx_1dz_2
\\
\noame
\quad\qquad \displaystyle =-{1\over Pr}\int_{\widetilde \Omega^\varepsilon}(\tilde{\bf  u}^\varepsilon\cdot \nabla_{\eta_\varepsilon}) \tilde {\bf u}^\varepsilon\tilde \varphi\,dx_1dz_2+{2N\over 1-N}\int_{\widetilde \Omega^\varepsilon}\nabla_{\eta_\varepsilon}^\perp \tilde w^\varepsilon\cdot \tilde \varphi \,dx_1dz_2
\\
\noame
\displaystyle \quad\qquad\qquad\qquad\qquad\quad
+Ra\int_{\widetilde \Omega^\varepsilon}\tilde T^\varepsilon ({\bf e}_2\cdot \tilde \varphi)\,dx_1dz_2+{1\over \eta_\varepsilon^2}\int_{\widetilde \Omega^\varepsilon}f_1 ({\bf e}_1\cdot \tilde \varphi)\,dx_1dz_2,
\end{array}
\end{equation}
\begin{equation}\label{form_var_1_wdil}
 \begin{array}{l}
\displaystyle L\int_{\widetilde \Omega^\varepsilon}\nabla_{\eta_\varepsilon}\tilde  w^\varepsilon\cdot \nabla_{\eta_\varepsilon} \tilde \psi\,dx_1dz_2 + {4N\over 1-N}\int_{\widetilde \Omega^\varepsilon}\tilde w^\varepsilon \tilde \psi\,dx_1dz_2\\
\noame
\displaystyle \qquad =-{M\over Pr}\int_{\widetilde \Omega^\varepsilon}( \tilde{\bf  u}^\varepsilon \cdot \nabla_{\eta_\varepsilon}\tilde  w^\varepsilon)\tilde \psi\,dx_1dz_2+{2N\over 1-N}\int_{\widetilde \Omega^\varepsilon} {\rm rot}_{\eta_\varepsilon}( \tilde{\bf  u}^\varepsilon) \tilde \psi\,dxdz_2+{1\over\eta_\varepsilon^2}\int_{\widetilde \Omega^\varepsilon}g\, \tilde \psi\,dx_1dz_2,
\end{array}
\end{equation}

\begin{equation}\label{form_var_1_Tdil}
\begin{array}{l}
\displaystyle  \int_{\widetilde \Omega^\varepsilon}\nabla_{\eta_\varepsilon}\tilde  T^\varepsilon\cdot \nabla_{\eta_\varepsilon}\tilde  \phi\,dx_1dz_2+  k\int_{\widetilde \Gamma^\varepsilon_1}\tilde T^\varepsilon \tilde \phi\,d\sigma\\
\noame
\qquad\displaystyle
= -  \int_{\widetilde \Omega^\varepsilon}(\tilde{\bf  u}^\varepsilon\cdot \nabla_{\eta_\varepsilon})\tilde T^\varepsilon\tilde \phi\,dx_1dz_2+D \int_{\widetilde \Omega^\varepsilon}\nabla_{\eta_\varepsilon}^\perp \tilde w^\varepsilon\cdot \nabla_{\eta_\varepsilon} \tilde T^\varepsilon\tilde \phi\,dx_1dz_2+  k \int_{\widetilde \Gamma^\varepsilon_1}G^\varepsilon\tilde \phi\,d\sigma
.\end{array}
\end{equation}
for every $\tilde \varphi\in H^1_0(\widetilde \Omega^\ep)^2, \tilde \psi\in H^1_0( \widetilde \Omega^\ep )$ and $\tilde \phi\in H^1(\widetilde \Omega^\ep)$  such that $\tilde \phi=0\ \hbox{on}\ \partial   \widetilde \Omega^\ep \setminus \widetilde \Gamma^\varepsilon_1$, and $\tilde\varphi(x_1,z_2)=\varphi(x_1,\eta_\varepsilon z_2)$, $\tilde\psi(x_1,z_2)=\psi(x_1,\eta_\varepsilon z_2)$ and $\tilde\phi(x_1,z_2)=\phi(x_1,\eta_\varepsilon z_2)$ for a.e. $(x_1,z_2)\in \widetilde\Omega^\varepsilon$.\\

\noindent Next, according previous estimates of the unfolding functions, we consider as test functions in (\ref{form_var_1_udil})-(\ref{form_var_1_Tdil}) the following ones
$$\begin{array}{l}
\varphi^\ep(x_1,z_2)=  \eta_\ep^2\varphi(x_1,x_1/\ep,z_2)\quad\hbox{with}\quad  \varphi(x_1,z)\in \mathcal{D}(\omega;C_{\#}^\infty(Z)^2),\\
\noame
\psi^\ep(x_1,z_2)= \eta_\ep^2\psi(x_1,x_1/\ep,z_2)\quad\hbox{with}\quad   \psi(x_1,z)\in \mathcal{D}(\omega;C_{\#}^\infty(Z)),\\
\noame
\phi^\ep(x_1,z_2)=  \phi(x_1,x_1/\ep,z_2)\quad\hbox{with}\quad   \phi(x_1,z)\in \mathcal{D}(\omega;C_{\#}^\infty(Z)).
\end{array}$$
Taking into account this, the formulation (\ref{form_var_1_udil}) reads
\begin{equation}\label{form_var_1_changvar_proof1}
 \begin{array}{l}
\displaystyle {1\over 1-N}\int_{\widetilde\Omega_\ep}\eta_\ep^2\nabla_{\eta_\ep}\tilde{\bf  u}^\varepsilon\cdot \nabla_{\eta_\ep} \varphi^\ep\,dx_1dz_2
 +{1\over Pr}\int_{\widetilde \Omega^\ep}\eta_\ep^2\partial_{x_1} p^\ep_0(x_1)\,\varphi_1^\ep\,dx_1dz_2-{1\over Pr}\int_{\widetilde \Omega^\ep}\eta_\ep^2\tilde p^\ep_1\,{\rm div}_{\eta_\ep}\,\varphi^\ep\,dx_1dz_2
\\
\noame
\qquad\qquad\qquad \displaystyle =-{1\over Pr}\int_{\widetilde\Omega^\varepsilon}\eta_\ep^2( \tilde{\bf  u}^\varepsilon\cdot \nabla_{\eta_\ep})\tilde{\bf  u}^\varepsilon\varphi^\ep\,dx_1dz_2 +{2N\over 1-N}\int_{\widetilde\Omega^\varepsilon}\eta_\ep^2\nabla_{\eta_\ep}^\perp \tilde w^\varepsilon\varphi^\ep \,dx_1dz_2
\\
\noame
\displaystyle \qquad\qquad \qquad\quad
+Ra\int_{\widetilde\Omega^\varepsilon}\eta_\ep^2\tilde T^\varepsilon({\bf e}_2\cdot \varphi^\ep)\,dx_1dz_2+\int_{\widetilde\Omega^\varepsilon}f_1 ({\bf e}_1\cdot \varphi^\ep)\,dx_1dz_2,
\end{array}\end{equation}
the formulation (\ref{form_var_1_wdil}) reads
\begin{equation}\label{form_var_1_changvar_proof2}
\begin{array}{l}
\displaystyle L\int_{\widetilde \Omega^\varepsilon}\eta_\ep^2\nabla_{\eta_\ep} \tilde w^\varepsilon\cdot \nabla_{\eta_\ep}  \psi^\ep\,dx_1dz_2 + {4N\over 1-N}\int_{\widetilde \Omega^\varepsilon}\eta_\ep^2\tilde w^\varepsilon   \psi^\ep\,dx_1dz_2\\
\noame
\displaystyle \qquad =-{M\over Pr}\int_{\widetilde \Omega^\varepsilon}\eta_\ep^2(\tilde{\bf  u}^\varepsilon \cdot \nabla_{\eta_\ep} \tilde w^\varepsilon) \psi^\ep\,dx_1dz_2+{2N\over 1-N}\int_{\widetilde \Omega^\varepsilon} \eta_\ep^2{\rm rot}_{\eta_\ep}(\tilde{\bf  u}^\varepsilon) \psi^\ep\,dx_1dz_2+\int_{\widetilde \Omega^\varepsilon}g\, \psi^\ep\,dx_1dz_2,
\end{array}
\end{equation}
and the formulation (\ref{form_var_1_Tdil}) reads reads
\begin{equation}\label{form_var_1_changvar_proof3}
 \begin{array}{l}
\displaystyle \int_{\widetilde \Omega^\varepsilon}\nabla_{\eta_\ep} \tilde T^\varepsilon\cdot \nabla_{\eta_\ep}   \phi^\ep\,dx_1dz_2+  k\int_{\widetilde \Gamma_1^\varepsilon}\tilde T^\varepsilon  \phi^\ep\,d\sigma\\
\noame
\qquad\displaystyle
=- \int_{\widetilde\Omega^\varepsilon}(\tilde{\bf  u}^\varepsilon\cdot \nabla_{\eta_\ep})\tilde T^\varepsilon \phi^\ep\,dx_1dz_2
 +D\int_{\widetilde \Omega^\varepsilon}\nabla_{\eta_\ep}^\perp \tilde w^\varepsilon\cdot \nabla_{\eta_\ep}\tilde T^\varepsilon  \phi^\ep\,dx_1dz_2 +   k \int_{\widetilde \Gamma_1^\varepsilon}G^\varepsilon   \phi^\ep\,d\sigma.
\end{array}
\end{equation}

\noindent By the unfolding change of variables  (see \cite{Anguiano_SG}, \cite{grauMicRough} for more details), we get
$$ \begin{array}{l}
\displaystyle{1\over Pr}\eta_\ep^2\int_{\widetilde\Omega^\varepsilon}( \tilde {\bf u}^\varepsilon\cdot \nabla_{\eta_\ep}) \tilde {\bf u}^\varepsilon\varphi^\ep\,dx_1dz_2\medskip\\
\displaystyle=-{\eta_\ep^2\over Pr}\int_{\widetilde\Omega^\varepsilon}\tilde  {\bf u}^\varepsilon \tilde\otimes  \tilde  {\bf u}^\varepsilon  \partial_{x_1}\varphi\,dx_1dz_2+{\eta_\varepsilon\over Pr}\left(\int_{\widetilde\Omega^\varepsilon}\partial_{z_2}{\tilde u}^\varepsilon_2 \tilde {\bf u}^\varepsilon\cdot \varphi\,dx_1dz_2+
\int_{\widetilde\Omega^\varepsilon} \tilde u^\varepsilon_2\partial_{z_2}  \tilde{\bf  u}^\varepsilon\cdot \varphi\,dx_1dz_2\right) \medskip\\
\displaystyle=-{\eta_\ep^2\varepsilon^{-1}\over Pr}\int_{\omega\times Z}\hat  {\bf u}^\varepsilon \tilde\otimes  \hat  {\bf u}^\varepsilon \cdot \partial_{z_1}\varphi\,dx_1dz+{\eta_\varepsilon\over Pr}\left(\int_{\omega\times Z}\partial_{z_2}\hat { u}^\varepsilon_2 \hat {\bf u}^\varepsilon\cdot \varphi\,dx_1dz+
\int_{\omega\times Z} \hat u^\varepsilon_2\partial_{z_2}\hat {\bf u}^\varepsilon\cdot \varphi\,dx_1dz\right)+O_\varepsilon,
\end{array}$$
where the operation $\tilde\otimes$ is defined by (\ref{tildeotimes}) and $O_\varepsilon$ tends to zero. With this and by applying the unfolding change of variables to the rest of the terms in (\ref{form_var_1_changvar_proof1}), we get
\begin{equation}\label{form_var_1_changvar_hat1}
 \begin{array}{l}
\displaystyle {1\over 1-N}\int_{\omega\times Z}\eta_\ep^2\ep^{-2}\partial_{z_1}  \hat{\bf  u}^\varepsilon\cdot \partial_{z_1} \varphi\,dx_1dz+{1\over 1-N}\int_{\omega\times Z}\partial_{z_2} \hat{\bf  u}^\varepsilon\cdot \partial_{z_2} \varphi\,dx_1dz\\
\noame\displaystyle\qquad +{1\over Pr}\int_{\omega\times Z}\eta_\ep^2\ep^{-1}\partial_{z_1}\hat p^\ep_0\, \varphi_1 \,dx_1dz-{1\over Pr}\int_{\omega\times Z}\eta_\ep^2\ep^{-1}\hat p^\varepsilon_1\,\partial_{z_1}\, \varphi_1\,dx_1dz-{1\over Pr}\int_{\omega\times Z}\eta_\ep\hat p^\varepsilon_1\,\partial_{z_2}\, \varphi_2\,dx_1dz
\\
\noame
\displaystyle\qquad
={\eta_\ep^2\varepsilon^{-1}\over Pr}\int_{\omega\times Z}\hat  {\bf u}^\varepsilon \tilde\otimes  \hat  {\bf u}^\varepsilon \cdot \partial_{z_1}\varphi\,dx_1dz-{\eta_\varepsilon\over Pr}\left(\int_{\omega\times Z}\partial_{z_2}{\hat u}^\varepsilon_2 \hat {\bf u}^\varepsilon\cdot \varphi\,dx_1dz+
\int_{\omega\times Z} \eta_\ep^2\hat u^\varepsilon_2\partial_{z_2}  \hat {\bf u}^\varepsilon\varphi\,dx_1dz\right)
\\
\noame
\displaystyle\qquad +
{2N\over 1-N}\int_{\omega\times Z} \partial_{z_2} \hat w^\ep \varphi_1 \,dx_1dz
 -{2N\over 1-N}\int_{\omega\times Z} \eta_\ep^2\ep^{-1}\partial_{z_1}\hat w^\ep \varphi_2 \,dx_1dz
\\
\noame\displaystyle\qquad+Ra\int_{\omega\times Z}\eta_\ep^2\hat T^\varepsilon ({\bf e}_2\cdot \varphi)\,dx_1dz+ \int_{\omega\times Z}f_1 ({\bf e}_1\cdot \varphi )\,dx_1dz+O_\ep,
\end{array}\end{equation}
 with $O_\varepsilon$ devoted to tend to zero.\\

\noindent Analogously, applying the unfolding change of variables to the equation  (\ref{form_var_1_changvar_proof2}), we get
\begin{equation}\label{form_var_1_changvar_hat2}\begin{array}{l}
\displaystyle L\int_{\omega\times Z} \eta_\ep^2\ep^{-2}\partial_{z_1} \hat w^\varepsilon\, \partial_{z_1} \psi\,dx_1dz+L\int_{\omega\times Z}  \partial_{z_2} \hat w^\varepsilon\, \partial_{z_2} \psi\,dx_1dz + {4N\over 1-N}\int_{\omega\times Z} \eta_\ep^2\hat w^\varepsilon  \psi\,dx_1dz\\
\noame
\displaystyle \qquad = -{M\over Pr}\int_{\omega\times Z}\eta_\ep^2 \ep^{-1} \hat u^\varepsilon_1 \partial_{z_1}\hat w^\ep   \psi\,dx_1dz-{M\over Pr}\int_{\omega\times Z} \eta_\ep \hat u^\varepsilon_2\partial_{z_2}\hat w^\ep   \psi\,dx_1dz\\
\noame\displaystyle\qquad\quad
+{2N\over 1-N}\int_{\omega\times Z} \, \eta_\ep^2\ep^{-1}\partial_{z_1}\hat u^\varepsilon_2\, \psi\,dx_1dz-{2N\over 1-N}\int_{\omega\times Z} \, \eta_\ep\partial_{z_2}\hat u^\varepsilon_1\psi\,dx_1dz\\
\noame\displaystyle\qquad\quad + \int_{\omega\times Z}g\, \psi\,dx_1dz + O_\ep,
\end{array}\end{equation}
with $O_\varepsilon$ devoted to tend to zero.\\

\noindent Finally, from (\ref{form_var_1_changvar_proof3}), we deduce
\begin{equation}\label{form_var_1_changvar_hat3}
\begin{array}{l}
\displaystyle \ep^{-2}\int_{\omega\times Z}  \partial_{z_1}\hat T^\varepsilon\, \partial_{z_1} \phi\,dx_1dz+\eta_\ep^{-2}\int_{\omega\times Z}  \partial_{z_2}\hat T^\varepsilon\, \partial_{z_2}  \phi\,dx_1dz\\
\noame
\qquad\displaystyle
=- \eta_\ep\int_{\omega\times Z}\Big(  \hat{\bf  u}^\varepsilon\cdot  \nabla_{\eta_\ep,\ep}\Big)\hat T^\varepsilon  \phi\,dx_1dz
\\
\noame\displaystyle\qquad +D\int_{\omega\times Z} \nabla^\perp_{\eta_\ep,\ep}\hat w^\ep\cdot \nabla_{\eta_\ep,\ep}(\eta_\ep^{-2} \hat T^\varepsilon) \phi\,dx_1dz+    k \int_{\omega\times \hat \Gamma_1}G\, \phi\,dx_1d\sigma+ O_\ep.
\end{array}\end{equation}
where we use the operators $\nabla_{\eta_\ep,\ep} =(\eta_\ep\ep^{-1} \partial_{z_1}, \partial_{z_2})$ and $\nabla^\perp_{\eta_\ep,\ep}=(\partial_{z_2},-\eta_\ep\ep^{-1} \partial_{z_1})$ and $O_\varepsilon$ is devoted to tend to zero.
\\

\noindent Here, we have used (\ref{trace_estim1_dil}) and (\ref{estim_sol_dil3}), which gives
$$\begin{array}{rl}\displaystyle
\left|k\int_{\widetilde \Gamma_1^\varepsilon}\tilde T^\varepsilon \phi^\varepsilon\,d\sigma\right|\leq C\|\tilde T^\varepsilon\|_{L^2(\tilde\Gamma_1^\varepsilon)}\leq C\eta_\ep\ep^{-{1\over 2}}\|\nabla_{\eta_\ep}\tilde T^\varepsilon\|_{L^2(\widetilde\Omega^\varepsilon)^2}\leq C\eta_\ep^2\ep^{-{1\over 2}}\to 0,
\end{array}
$$
and by the unfolding change of variables with respect to $x_1$, the periodicity of $h(z_1)$ and  $G(z_1)$, it holds
$$\begin{array}{rl}\displaystyle
k\int_{\widetilde \Gamma_1^\varepsilon}G^\varepsilon \phi^\varepsilon\,d\sigma=k\int_{\widetilde \Gamma_1^\varepsilon}G(x_1/\varepsilon)\phi^\varepsilon\,d\sigma= & \displaystyle \int_{\omega\times \hat \Gamma_1}G(z_1)\phi\,dx_1d\sigma +O_\varepsilon.
\end{array}
$$

\section{Homogenized model in the critical case}\label{sec:critical}
It corresponds to the critical case when the thickness of the domain is proportional to
the wavelength of the roughness, with $\lambda$ the proportionality constant, that is $\eta_\ep\approx\ep$, with $\eta_\ep/\ep\to \lambda$, $0<\lambda<+\infty$.\\

\noindent Let us introduce some notation which will be useful along this section. For a vectorial function ${\bf v}=(v_1,v_{2})$ and a scalar $w$, we introduce the operators $\nabla_\lambda$, $\Delta_\lambda$, ${\rm div}_\lambda$ and    by
$$\begin{array}{l}
(\nabla_\lambda {\bf v})_{i,1}=\lambda\partial_{z_1}v_i,\quad (\nabla_\lambda {\bf v})_{i,2}= \partial_{z_2}v_i\quad \hbox{for }i=1,2,
\\
\\
\Delta_\lambda {\bf v}=\lambda^2\partial^2_{z_1} {\bf v}+\partial^2_{z_2} {\bf v},\quad \nabla_\lambda w=(\lambda\partial_{z_1} w,\partial_{z_2}w)^t,\\
\\
{\rm div}_\lambda {\bf v}=\lambda\partial_{z_1} v_1+\partial_{z_2}v_{2},\quad    \displaystyle \nabla^\perp_\lambda w^\varepsilon=\left(\partial_{z_2} w,-\lambda \partial_{z_1}w\right)^t.
\end{array}$$
Next, we give some compactness results about the behavior of the  sequences $(  \tilde{\bf  U}^\ep, \tilde W^\ep, \tilde \theta^\ep, p^\ep_0, \tilde p^\ep_1)$ and the related unfolding functions $( \hat {\bf u}^\ep, \hat w^\ep, \hat T^\ep, \hat p^\ep_0, \hat p^\ep_1)$ satisfying the {\it a priori} estimates given in Lemma \ref{lemma_estimates2}, Corollary \ref{Lem_estim_pressures} and Lemma \ref{estimates_hat}  respectively.
\begin{lemma}\label{lem_asymp_crit}
For a subsequence of $\ep$ still denote by $\ep$, we have the following convergence results:
\begin{itemize}
\item[(i)] {\it (Velocity)} There exist $ \tilde{\bf  U}=(\tilde U_{1},\tilde U_{2})\in H^1(0,h_{\rm max};L^2(\omega)^2)$, with $ \tilde{\bf  U}=0$ on $z_2=\{0,h_{\rm max}\}$ and $\tilde U_{2}=0$, such that
\begin{eqnarray}
&\displaystyle   \tilde{\bf  U}^\ep\rightharpoonup   \tilde{\bf  U}\quad \hbox{in } H^1(0,h_{\rm max};L^2(\omega)^2),\label{conv_u_crit_tilde}\\
\noame
&\displaystyle \partial_{x_1}\left(\int_0^{h_{\rm max}}\tilde U_1(x_1,z_2)\,dz_2\right)=0\  \hbox{  in  }\omega,
&\label{div_x_crit_tilde}
 \end{eqnarray}
and  $ \hat {\bf u}=(\hat u_1, \hat u_{2})\in L^2(\omega;H^1_{\#}(Z))^2$, with $\hat {\bf u}=0$ on $z_2=\{0,h(z_1)\}$ such that it hold $\int_{Z} \hat{\bf  u}(x_1,z)dz$ $=\int_0^{h_{\rm max}} \tilde{\bf  U}(x_1,z_2)\,dz_2$ with $\int_{Z}\hat u_{2}(x_1,z)\,dz=0$, and moreover
\begin{eqnarray}
&
 \hat {\bf u}^\ep\rightharpoonup \hat {\bf u}\quad \hbox{in } L^2(\omega;H^1(Z)^2),&\label{conv_u_crit_hat}\\
\noame
&\displaystyle {\rm div}_\lambda  \hat {\bf u}=0\quad\hbox{  in  }\omega\times Z,&\label{div_crit_hat1}\\
\noame &\displaystyle
\partial_{x_1}\left(\int_{Z}\hat u_1(x_1,z)\,dz\right)=0\quad \hbox{  in  }\omega\,.
&\label{div_crit_hat2}
\end{eqnarray}
\item[(ii)] {\it (Microrotation)} There exist $\tilde W\in H^1(0,h_{\rm max};L^2(\omega))$, with $\tilde W=0$ on $z_2=\{0,h_{\rm max}\}$, such that
\begin{eqnarray}
&\displaystyle \tilde W^\ep\rightharpoonup \tilde W\quad \hbox{in } H^1(0,h_{\rm max};L^2(\omega)),\label{conv_w_crit_tilde}
\end{eqnarray}
and  $\hat w\in L^2(\omega;H^1_{\#}(Z))$, with $\hat w=0$ on $z_2=\{0,h(z_1)\}$ such that $\int_{Z}\hat w(x_1,z)dz=\int_0^{h_{\rm max}}\tilde W(x_1,z_2)\,dz_2$, and moreover
\begin{eqnarray}
\hat w^\ep\rightharpoonup \hat w\quad \hbox{in } L^2(\omega;H^1(Z)).&\label{conv_w_crit_hat}
\end{eqnarray}
\item[(iii)] {\it (Temperature)} There exist $\tilde \theta\in H^1(0,h_{\rm max};L^2(\omega))$, with $\tilde\theta=0$ on $z_2=\{0\}$, such that
\begin{eqnarray}
&\displaystyle \eta_\ep^{-{2}}\tilde \theta^\ep\rightharpoonup \tilde\theta\quad \hbox{in } H^1(0,h_{\rm max};L^2(\omega)),\label{conv_T_crit_tilde}
\end{eqnarray}
and  $\hat T\in L^2(\omega;H^1_{\#}(Z))$, with $\hat T=0$ on $z_2=\{0\}$, such that $\int_{Z}\hat T(x_1,z)dz=\int_0^{h_{\rm max}}\tilde \theta(x_1,z_2)\,dz_2$, and moreover
\begin{eqnarray}
\eta_\ep^{-2}\hat T^\ep\rightharpoonup \hat T\quad \hbox{in } L^2(\omega;H^1(Z)).&\label{conv_T_crit_hat}
\end{eqnarray}
\item[(iv)] {\it (Pressure)} There exist   three functions  $\tilde p\in L^2_0(\omega)\cap H^1(\omega)$, independent of $z_2$ with $\tilde p(i)=q_i$, $i=-1/2,1/2$,  $\hat p_0\in L^2(\omega;H^1_\#(Z'))$ and $\hat p_1\in L^2(\omega;L^2_\#(Z))$ such that
\begin{eqnarray}
&\displaystyle\eta_\varepsilon^{2}  p^\ep_0\rightharpoonup  \tilde p\quad \hbox{in } H^1(\omega),& \label{conv_P01_crit1}\\
\noame
&\displaystyle \eta_\varepsilon^{2}\ep^{-1} \partial_{z_1}\hat p^\ep_0\rightharpoonup  \partial_{z_1}\tilde p+\partial_{z_1} \hat p_0\quad \hbox{in } L^2(\omega;L^2(Z')),\quad \eta_\ep \hat p^\ep_1\rightharpoonup  \hat p_1\quad \hbox{in } L^2(\omega;L^2(Z)).& \label{conv_P01_crit2}
\end{eqnarray}
\end{itemize}
\end{lemma}
\begin{proof} We will only give some remarks and,  for more details,  we refer the reader to Lemmas 5.2-i) and 5.4-i) in \cite{Anguiano_SG}.\\

\noindent We start with the extension $\tilde{\bf  U}^\ep$. Estimates  (\ref{estim_sol_dil_ext_u}) imply the existence of $  \tilde {\bf U}\in H^1(0,h_{\rm max};L^2(\omega)^2)$ such that convergence (\ref{conv_u_crit_tilde}) holds, and the continuity of the trace applications from the space of $  \tilde{\bf  U}$ such that $\| \tilde{\bf  U}\|_{L^2}$ and $\|\partial_{z_2}  \tilde{\bf  U}\|_{L^2}$ are bounded to $L^2(\Gamma_1)$ and to $L^2(\Gamma_0)$ implies $ \tilde{\bf  U}=0$ on $\Gamma_1$ and $\Gamma_0$. Next, from the free divergence condition ${\rm div}_{\eta_\ep}( \tilde{\bf  U}^\ep)=0$, it can be deduced that $ \tilde U_{2}$ is independent of $z_2$, which together with the boundary conditions satisfied by $\tilde U_{2}$ on $z_2=\{0,h_{\rm max}\}$ implies that $\tilde U_{2}=0$. Finally, from the free divergence condition and the convergence (\ref{conv_u_crit_tilde}) of $  \tilde{\bf  U}^\ep$, it is straightforward the corresponding free divergence condition in a thin domain given in (\ref{div_x_crit_tilde}).\\

\noindent Concerning $\hat {\bf u}^\ep$, estimates given in (\ref{estim_u_hat}) imply the existence of $\hat u\in  L^2(\omega;H^1(Z)^3)$ such that convergence (\ref{conv_u_crit_hat}) holds. It can be proved the $Z'$-periodicity of $\hat u$, and applying the unfolding change of variables  to the free divergence condition ${\rm div}_{\eta_\ep}\tilde u_\ep=0$, passing to the limit, we get divergence condition (\ref{div_crit_hat1}). Finally,
 it can be proved that $\int_Z \hat {\bf u}(x_1,z)\,dz=\int_0^{h_{\rm max}} \tilde{\bf  U}(x_1,z_2)\,dz_2$ which together with $\tilde U_{2}=0$ implies $\int_0^{h_{\rm max}}\tilde U_{2}(x_1,z_2)\,dz_2=0$, and together with property (\ref{div_x_crit_tilde}) implies the divergence condition given in (\ref{div_crit_hat2}).\\

\noindent We continue proving $(ii)$.  From estimates (\ref{estim_sol_dil_ext_w}), convergence  (\ref{conv_w_crit_tilde}) and that $\tilde W=0$ on $z_2=\{0,h_{\rm max}\}$ are obtained straighfordward. The proofs of the convergence of $W^\ep$ and identity $\int_Z\hat w\,dz=\int_0^{h_{\rm max}}\tilde W\,dz_2$ are similar to the ones of $\hat {\bf u}^\ep$. \\

\noindent We continue with $(iii)$.  The proof is similar to $(ii)$, but we have to take into account estimates (\ref{estim_sol_dil_ext_T}) and that the dirichlet boundary condition for temperature is imposed on the bottom and not on the top. \\

\noindent We finish the proof with $(iv)$. From estimates of $ p_0^\ep$ and $\hat p_0^\ep$, and   the classical compactness result for the unfolding method for a bounded sequence in $H^1$, we get  convergences for (\ref{conv_P01_crit1}) and (\ref{conv_P01_crit2})$_1$. Estimate for $\hat p^\ep_1$ implies convergence (\ref{conv_P01_crit2})$_2$. From the boundary conditions of $\tilde p^\ep$ on $x_1=\{-1/2,1/2\}$, the decomposition of the pressure and the convergences of $\hat p^\ep_0$ and $\hat p^\ep_1$,  it holds the boundary conditions for $\tilde p$. Finally, since $\tilde p^\ep$ has mean value zero, from the decomposition of the pressure, we have
$$0=\int_{\widetilde\Omega^\ep}\eta^2_\ep\tilde p^\ep\,dx_1dz_2=\int_\omega h(x_1/\ep)\eta^2_\ep p^\ep_0\,dx_1+\int_{\widetilde\Omega^\ep}\eta^2_\ep\tilde p^\ep_1\,dx_1dz_2.$$
Taking into account that $h$ is $x_1$-periodic, the convergence of $\eta^2_\ep p^\ep_0$ to $\tilde p$ and that
$$\left|\int_{\widetilde\Omega^\ep}\eta^2_\ep\tilde p^\ep_1\,dx_1dz_2\right|\leq C\eta_\ep\to 0,$$
we get
$$ \int_{Z'}h\,dz_1\,\int_{\omega}\tilde p\,dx_1=0,$$
and so that $\tilde p$ has null mean value in $\omega$.

\end{proof}

\noindent Using previous convergences, in the following theorem we give the two-pressured homogenized
system satisfied by $( \hat{\bf u},\hat w, \tilde P, \hat T)$.

\begin{theorem}[Limit unfolded problems]\label{thm_limit_critical}
In the case $\eta_\ep\approx \ep$, with $\eta_\ep/\ep\to \lambda$, $0<\lambda<+\infty$, then the functions $ \hat{\bf u}, \hat w, \hat T$ and $\tilde p$ given in Lemma \ref{lem_asymp_crit} satisfy
\begin{itemize}
\item  $( \hat {\bf u}, \tilde p)\in L^2(\omega;H^1_{\#}(Z)^2)\times    (L^2_0(\omega)\cap H^1(\omega))$ is the unique solution of the two-pressure homogenized Stokes problem
\begin{equation}\label{hom_system_crit_u}
\left\{\begin{array}{rl}
\displaystyle
-{1\over 1-N}\Delta_{\lambda}   \hat{\bf u}+{1\over Pr}\nabla_{\lambda}\hat q=\left(f_1(x_1)-{1\over Pr}\partial_{x_1}\tilde  p(x_1)\right){\bf e}_1&\hbox{ in }\omega\times Z,\\
\noame
{\rm div}_{\lambda}  \hat{\bf u}=0&\hbox{ in }\omega\times Z,\\
\noame
 \hat{\bf u}=0&\hbox{ on }\omega \times (\hat \Gamma_0\cup \hat \Gamma_1),\\
\displaystyle \partial_{x_1}\left(\int_{Z}\hat u_1(x_1,z)\,dz\right)=0&\hbox{ in }\omega,\\
\noame
\displaystyle \int_Z\hat u_{2}\,dz=0&\hbox{ in }\omega,\\
\noame
\tilde p(i)=q_i&\ i=-1/2,1/2,\\
\noame
  \hat q(x_1,z)\in L^2(\omega;L^2_\#(Z)).&
\end{array}\right.
\end{equation}
\item   $\hat w\in L^2(\omega;H^1_{\#}(Z))$ is the unique solution of the Laplace problem
\begin{equation}\label{hom_system_crit_w}
\left\{\begin{array}{rl}
\displaystyle
-L\Delta_{\lambda} \hat w =g(x_1)&\hbox{ in }\omega\times Z,\\
\noame
\hat w=0&\hbox{ on }\omega \times (\hat \Gamma_0\cup \hat \Gamma_1),
\end{array}\right.
\end{equation}
\item   $\hat T\in L^2(\omega;H^1_{\#}(Z))$ is the unique solution of the nonlinear problem
\begin{equation}\label{hom_system_crit_T}
\left\{\begin{array}{rl}
 -\Delta_{\lambda} \hat T -D\nabla_{\lambda}^\perp\hat w\cdot \nabla_{\lambda}\hat T=0&\hbox{ in }\omega\times Z,\\
\noame
\hat T=0&\hbox{ on }z_2=\omega \times \hat \Gamma_0,\\
\noame
\nabla_\lambda \hat T\cdot {\bf n}=k\, G(z_1)&\hbox{ on }\omega \times \hat \Gamma_1.
\end{array}\right.
\end{equation}
\end{itemize}
\end{theorem}
\begin{proof}  We only have to prove   (\ref{hom_system_crit_u})$_{1}$, (\ref{hom_system_crit_w})$_{1}$ and (\ref{hom_system_crit_T})$_{1,3}$. The rest  follows from Lemma \ref{lem_asymp_crit}. We divide the proof  in three steps, where we are going to pass to the limit in variational formulation (\ref{form_var_1_changvar_hat1})-(\ref{form_var_1_changvar_hat3}), taking into account that $\eta_\ep/\ep\to \lambda$, $0<\lambda<+\infty$.\\

\noindent {\it Step 1.}  To prove   (\ref{hom_system_crit_u})$_{1}$,  we consider (\ref{form_var_1_changvar_hat1}) with $\varphi$ replaced by $\bar \varphi^\ep=(\lambda(\varepsilon/\eta_\varepsilon)\varphi_1,\varphi_2)$ with $\varphi=(\varphi_1, \varphi_2)\in\mathcal{D}(\omega;C^\infty_{\#}(Z)^2)$. This gives the following variational formulation:
\begin{equation}\label{Form_var_crit_pass_limit}
 \begin{array}{l}
\displaystyle {1\over 1-N}\int_{\omega\times Z}\eta_\ep\ep^{-1}\lambda\partial_{z_1}{\hat u}^\varepsilon_1\, \partial_{z_1}  \varphi_1\,dx_1dz+{1\over 1-N}\int_{\omega\times Z}\eta_\ep^2\ep^{-2}\partial_{z_1}{\hat u}^\varepsilon_2\, \partial_{z_1}  \varphi_2\,dx_1dz
\\
\noame
\displaystyle   +{1\over 1-N}\int_{\omega\times Z} \lambda(\varepsilon/\eta_\varepsilon)\partial_{z_2}{ \hat u_1}^\varepsilon\, \partial_{z_2} \varphi_1\,dx_1dz+{1\over 1-N}\int_{\omega\times Z} \partial_{z_2}{ \hat u_2}^\varepsilon\, \partial_{z_2} \varphi_2\,dx_1dz
\\
\noame\displaystyle     +{1\over Pr}\lambda   (\varepsilon/\eta_\varepsilon)\int_{\omega\times Z}\eta_\ep^2\ep^{-1}\partial_{z_1}\hat p^\ep_0\,\partial_{z_1}\varphi_1 \,dx_1dz
-{1\over Pr}\int_{\omega\times Z}\lambda\eta_\ep\hat p_1^\varepsilon\,\partial_{z_1}\, \varphi_1\,dx_1dz-{1\over Pr}\int_{\omega\times Z}\eta_\ep\hat p_1^\varepsilon\,\partial_{z_2}\, \varphi_2\,dx_1dz
\\
\noame
\displaystyle
 = {\eta_\ep^2\varepsilon^{-1}\over Pr}\int_{\omega\times Z}{\hat  u}^\varepsilon \tilde\otimes  {\hat  u}^\varepsilon  \,\partial_{z_1}\bar \varphi^\varepsilon\,dx_1dz-{\eta_\varepsilon \over Pr}\left(\int_{\omega\times Z}\partial_{z_2}{\hat u}^\varepsilon_2 \hat{\bf u}^\varepsilon\cdot \bar \varphi^\varepsilon\,dx_1dz+
\int_{\omega\times Z} \hat u^\varepsilon_2\partial_{z_2}  \hat{\bf u}^\varepsilon\bar \varphi^\varepsilon\,dx_1dz\right)
\\
\noame
\displaystyle  
+{2N\over 1-N}\int_{\omega\times Z} \eta_\varepsilon \lambda(\varepsilon/\eta_\varepsilon)\partial_{z_2} \hat w^\ep \varphi_1 \,dx_1dz
 -{2N\over 1-N}\int_{\omega\times Z} \eta_\ep^2\ep^{-1}\partial_{z_1}\hat w^\ep \varphi_2 \,dx_1dz
\\
\noame\displaystyle  
+Ra\int_{\omega\times Z}\eta_\ep^2\hat T^\varepsilon\, \varphi_2\,dx_1dz + \int_{\omega\times Z} \lambda(\varepsilon/\eta_\varepsilon)f_1\, \varphi_1\,dx_1dz+O_\ep,
\end{array}
\end{equation}
where $O_\varepsilon$ is devoted to tends to zero when $\varepsilon\to 0$. Below, let us pass to the limit when $\varepsilon$ tends to zero in each term of (\ref{Form_var_crit_pass_limit}):
\begin{itemize}
\item For the first fourth terms in the left-hand side of (\ref{Form_var_crit_pass_limit}), taking into account convergence (\ref{conv_u_crit_hat}) and that that $\eta_\varepsilon/\varepsilon\to \lambda$ and $\lambda(\varepsilon/\eta_\varepsilon)\to 1$, we get that
$$\begin{array}{l}\displaystyle
{1\over 1-N}\int_{\omega\times Z}\eta_\ep\ep^{-1}\lambda\partial_{z_1}{\hat u}^\varepsilon_1\, \partial_{z_1}  \varphi_1\,dx_1dz+{1\over 1-N}\int_{\omega\times Z}\eta_\ep^2\ep^{-2}\partial_{z_1}{\hat u}^\varepsilon_2\, \partial_{z_1}  \varphi_2\,dx_1dz
\\
\noame
\displaystyle
+{1\over 1-N}\int_{\omega\times Z} \lambda(\varepsilon/\eta_\varepsilon)\partial_{z_2}{ \hat u_1}^\varepsilon\, \partial_{z_2} \varphi_1\,dx_1dz+{1\over 1-N}\int_{\omega\times Z} \partial_{z_2}{ \hat u_2}^\varepsilon\, \partial_{z_2} \varphi_2\,dx_1dz
\end{array}$$
converges to
$${1\over 1-N}\int_{\omega\times Z}\lambda^2\partial_{z_1}\hat{\bf  u}^\varepsilon\cdot\partial_{z_1}  \varphi\,dx_1dz+{1\over 1-N}\int_{\omega\times Z}\partial_{z_2}  \hat{\bf  u}^\varepsilon\cdot \partial_{z_2}  \varphi\,dx_1dz=
 {1\over 1-N}\int_{\omega\times Z}\nabla_{\lambda}  \hat{\bf  u}\cdot \nabla_{\lambda} \varphi\,dx_1dz.
$$

\item For the fifth  to  eighth terms in the left hand side of (\ref{Form_var_crit_pass_limit}), taking into account that $\lambda(\varepsilon/\eta_\varepsilon)\to 1$ and the convergences   (\ref{conv_P01_crit1}) and (\ref{conv_P01_crit2}), we have the following convergences
$$\begin{array}{l}
\displaystyle
{1\over Pr}\lambda   (\varepsilon/\eta_\varepsilon)\int_{\omega\times Z}\eta_\ep^2\ep^{-1}\partial_{z_1}\hat p^\ep_0\,\varphi_1 \,dx_1dz\to {1\over Pr}\int_{\omega\times Z}(\partial_{x_1}\tilde p  +\partial_{z_1}\hat p_0)\,\varphi_1\,dx_1dz,
\end{array}
$$
$$-{1\over Pr}\int_{\omega\times Z}\lambda\eta_\ep\hat p^\varepsilon_1\,\partial_{z_1}\, \varphi_1\,dx_1dz-{1\over Pr}\int_{\omega\times Z}\hat \eta_\ep p^\varepsilon_1\,\partial_{z_2}\, \varphi_2\,dx_1dz\to -{1\over Pr}\int_{\omega\times Z} \hat p_1\,{\rm div}_\lambda\varphi\,dx_1dz.$$

\item For the first three terms in the right-hand side of (\ref{Form_var_crit_pass_limit}), by taking into account the estimates (\ref{estim_u_hat}), we get
$$\left|{\eta_\ep^2\varepsilon^{-1}\over Pr}\int_{\omega\times Z}\hat  {\bf u}^\varepsilon \tilde\otimes  \hat  {\bf u}^\varepsilon  \,\partial_{z_1}\bar \varphi^\varepsilon\,dx_1dz\right|\leq  C\eta_\ep^2\varepsilon^{-1}\|  \hat{\bf  u}^\varepsilon\|^2_{L^2(\omega\times Z)^2}\|\partial_{z_1} \varphi\|_{L^\infty(\omega\times Z)^2}\leq C\eta_\ep^2\varepsilon^{-1}\to 0,$$
and
$$\begin{array}{l}\displaystyle
\left|{\eta_\varepsilon \over Pr}\left(\int_{\omega\times Z}\partial_{z_2}{\hat u}^\varepsilon_2  \hat{\bf  u}^\varepsilon\cdot \bar \varphi^\varepsilon\,dx_1dz+
\int_{\omega\times Z} \hat u^\varepsilon_2\partial_{z_2}  \hat{\bf  u}^\varepsilon\bar \varphi^\varepsilon\,dx_1dz\right)\right|\\
\noame \displaystyle\qquad \leq \eta_\varepsilon\| \hat {\bf u}^\varepsilon\|_{L^2(\omega\times Z)^2}\|\partial_{z_2} \hat{\bf  u}^\varepsilon\|_{L^2(\omega\times Z)^2}\|\varphi\|_{L^\infty(\omega\times Z)^2}\\
\noame
\qquad \leq  C\eta_\varepsilon\to 0.
\end{array}
$$
\noindent Then, we deduce that the convective terms satisfy
$${\eta_\ep^2\varepsilon^{-1}\over Pr}\int_{\omega\times Z}\hat {\bf  u}^\varepsilon \tilde\otimes  \hat {\bf  u}^\varepsilon  \,\partial_{z_1}\bar \varphi^\varepsilon\,dx_1dz-{\eta_\varepsilon \over Pr}\left(\int_{\omega\times Z}\partial_{z_2}{\hat u}^\varepsilon_2  \hat{\bf  u}^\varepsilon\cdot \bar \varphi^\varepsilon\,dx_1dz+
\int_{\omega\times Z} \hat u^\varepsilon_2\partial_{z_2}  \hat{\bf  u}^\varepsilon\bar \varphi^\varepsilon\,dx_1dz\right)\to 0.$$

\item For the fourth and fifth terms in the right-hand side of (\ref{Form_var_crit_pass_limit}), by taking into account convergence (\ref{conv_w_crit_hat}), we have
$${2N\over 1-N}\int_{\omega\times Z} \eta_\varepsilon \lambda(\varepsilon/\eta_\varepsilon)\partial_{z_2} \hat w^\ep \varphi_1 \,dx_1dz
 -{2N\over 1-N}\int_{\omega\times Z} \eta_\ep^2\ep^{-1}\partial_{z_1}\hat w^\ep \varphi_2 \,dx_1dz\to 0.$$
\item For the sixth term in the right-hand side of (\ref{Form_var_crit_pass_limit}), by taking into account convergence (\ref{conv_T_crit_hat}), we get
$$Ra\int_{\omega\times Z}\eta_\ep^2\hat T^\varepsilon \,\varphi_2\,dx_1dz\to 0.$$
\item For the last term in the right-hand side of (\ref{Form_var_crit_pass_limit}), by taking into account that $\lambda(\varepsilon/\eta_\varepsilon)\to 1$, we get
$$\int_{\omega\times Z} \lambda(\varepsilon/\eta_\varepsilon)f_1 \,\varphi_1\,dx_1dz\to   \int_{\omega\times Z} f_1 \,\varphi_1\,dx_1dz.$$
\end{itemize}
Therefore, by previous convergences, and taking  $\hat q:=\hat p_0/\lambda+\hat p_1\in L^2(\omega;L^2_\#(Z))$, we deduce that the limit variational formulation is given by the following one
\begin{equation}\label{form_var_1_changvar_hat1_limit2}
 \begin{array}{l}
\displaystyle {1\over 1-N}\int_{\omega\times Z}\nabla_{\lambda} \hat{\bf  u}\cdot \nabla_{\lambda} \varphi\,dx_1dz +{1\over Pr}\int_{\omega\times Z}  \partial_{x_1}\tilde p\,({\bf e_1}\cdot \varphi)\,dx_1dz-{1\over Pr}\int_{\omega\times Z} \hat q\,{\rm div}_\lambda\varphi\,dx_1dz\\
\noame
\displaystyle
  =
 \int_{\omega\times Z}f_1 \,({\bf e}_1\cdot \varphi)\,dx_1dz.
\end{array}\end{equation}
\noindent By density, (\ref{form_var_1_changvar_hat1_limit2}) holds for every function $\varphi\in L^2(\omega;H^1_{\#}(Z)^2)$ and is equivalent to the system (\ref{hom_system_crit_u})$_1$. We also remark that (\ref{form_var_1_changvar_hat1_limit2}) admits a unique solution, and then the convergence is for the   complete sequences of the unknowns.\\

\noindent {\it Step 2.} Next, we  prove (\ref{hom_system_crit_w})$_{1}$. Let us pass to the limit when $\varepsilon$ tends to zero in each term of the variational formulation  (\ref{form_var_1_changvar_hat2}):
\begin{itemize}
\item For the first two terms in the left-hand side of  (\ref{form_var_1_changvar_hat2}), by using convergence (\ref{conv_w_crit_hat}) and $\eta_\varepsilon/\ep\to \lambda$, we get
$$\displaystyle L\int_{\omega\times Z} \eta_\ep^2\ep^{-2}\partial_{z_1} \hat w^\varepsilon\, \partial_{z_1} \psi\,dx_1dz+L\int_{\omega\times Z}  \partial_{z_2} \hat w^\varepsilon\, \partial_{z_2} \psi\,dx_1dz \to  L \int_{\omega\times Z}\nabla_{\lambda} \hat w \cdot \nabla_{\lambda} \psi\,dx_1dz.
$$
\item For the third term in the left-hand side of  (\ref{form_var_1_changvar_hat2}),   by using convergence (\ref{conv_w_crit_hat}), we have
$${4N\over 1-N}\int_{\omega\times Z}\eta_\ep^2 \hat w^\varepsilon  \psi\,dx_1dz\to 0.$$
\item For the first two terms in the right-hand side of   (\ref{form_var_1_changvar_hat2}), by using estimates (\ref{estim_u_hat}) and (\ref{estim_w_hat}), we get
$$\begin{array}{l}\displaystyle \left|-{M\over Pr}\int_{\omega\times Z} \eta_\ep^2\ep^{-1} \hat u^\varepsilon_1 \partial_{z_1}\hat w^\ep   \psi\,dx_1dz\right|\leq C\eta_\ep^2\ep^{-1}\| \hat{\bf  u^\varepsilon}\|_{L^2(\omega\times Z)^2}\|\partial_{z_1}\hat w^\ep\|_{L^2(\omega\times Z)}\leq C\eta_\ep,
\\
\noame
\displaystyle
\left|-{M\over Pr}\int_{\omega\times Z} \eta_\ep \hat u^\varepsilon_2\partial_{z_2}\hat w^\ep   \psi\,dx_1dz\right|\leq
C\eta_\ep\|\hat{\bf  u^\varepsilon}\|_{L^2(\omega\times Z)^2}\|\partial_{z_2}\hat w^\ep\|_{L^2(\omega\times Z)}\leq C\eta_\ep.
\end{array}$$
Thus, we get
$$-{M\over Pr}\int_{\omega\times Z} \eta_\ep^2\ep^{-1} \hat u^\varepsilon_1 \partial_{z_1}\hat w^\ep   \psi\,dx_1dz-{M\over Pr}\int_{\omega\times Z} \eta_\ep \hat u^\varepsilon_2\partial_{z_2}\hat w^\ep   \psi\,dx_1dz\to 0.$$
\item For the third and fourth terms in the right-hand side of  (\ref{form_var_1_changvar_hat2}), by using estimates (\ref{estim_u_hat}) and (\ref{estim_w_hat}), we get
$$\begin{array}{l}\displaystyle
\left|{2N\over 1-N}\int_{\omega\times Z} \, \eta_\ep^2\ep^{-1}\partial_{z_1}\hat u^\varepsilon_2\, \psi\,dx_1dz\right|\leq C\eta_\ep^2\ep^{-1}\|\partial_{z_1} \hat{\bf  u}^\varepsilon\|_{L^2(\omega\times Z)^2}\leq C\eta_\ep,\\
\noame
\displaystyle
\left|-{2N\over 1-N}\int_{\omega\times Z} \, \eta_\ep\partial_{z_2}\hat u^\varepsilon_1\psi\,dx_1dz\right|\leq C\eta_\ep\|\partial_{z_2}\hat{\bf  u}^\varepsilon\|_{L^2(\omega\times Z)^2}\leq C\eta_\ep.
\end{array}$$
Thus, we have
$${2N\over 1-N}\int_{\omega\times Z} \, \eta_\ep^2\ep^{-1}\partial_{z_1}\hat u^\varepsilon_2\, \psi\,dx_1dz-{2N\over 1-N}\int_{\omega\times Z} \, \eta_\ep\partial_{z_2}\hat u^\varepsilon_1\psi\,dx_1dz\to 0.$$
\end{itemize}
Then, from the above convergences, we get that the limit variational formulation for $\hat w$ is given by
\begin{equation}\label{form_var_1_changvar_proof2_limit}\begin{array}{l}
\displaystyle L \int_{\omega\times Z}\nabla_{\lambda} \hat w \cdot \nabla_{\lambda} \psi\,dx_1dz  =\int_{\omega\times Z}g\, \psi\,dx_1dz.
\end{array}\end{equation}
By density, (\ref{form_var_1_changvar_proof2_limit}) holds for every function $\psi$ in $L^2(\omega;H^1_{\#}(Z))$ and is equivalent to problem (\ref{hom_system_crit_w}). We   remark that (\ref{form_var_1_changvar_proof2_limit}) admits a unique solution, and then the complete sequence $\hat w^\ep$ converges to the unique solution $\hat w(x_1,z)$. \\

\noindent Before passing to the next step, we need to prove that $\nabla_{\eta_\ep,\ep}\hat w^\ep$ converges strongly to $\nabla_\lambda \hat w$ in $L^2(\omega\times Z)^2$. To do this, we take $\hat w^\ep$ as test function in (\ref{form_var_1_changvar_hat2}) and $\hat w$ in (\ref{form_var_1_changvar_proof2_limit}). Then, it is easy to prove that
$$
\lim_{\ep\to 0}\displaystyle \int_{\omega\times Z}  |\nabla_{\eta_\ep,\ep}\hat w^\varepsilon|^2\,dx_1dz={1\over L}\int_{\omega\times Z}g\, \hat w\,dx_1dz=\int_{\omega\times Z}  |\nabla_{\lambda}\hat w|^2\,dx_1dz
$$
This together with the weak convergence of $\nabla_{\eta_\ep,\ep}\hat w^\ep$ to $\nabla_\lambda \hat w$ in $L^2(\omega\times Z)^2$, it gives the desired strong convergence.\\

\noindent {\it Step 3.} To prove  (\ref{hom_system_crit_T})$_{1,4}$, we take into account that the variational formulation (\ref{form_var_1_changvar_hat3}) can be written as follows
\begin{equation}\label{form_var_1_changvar_hat3pass}
\begin{array}{l}
\displaystyle \eta_\ep^2\ep^{-2}\int_{\omega\times Z}  \eta_\ep^{-2}\partial_{z_1}\hat T^\varepsilon\, \partial_{z_1} \phi\,dx_1dz+\eta_\ep^{-2}\int_{\omega\times Z}  \partial_{z_2}\hat T^\varepsilon\, \partial_{z_2}  \phi\,dx_1dz\\
\noame
\qquad\displaystyle
=- \eta_\ep\int_{\omega\times Z}\Big(\hat {\bf u}^\varepsilon\cdot  \nabla_{\eta_\ep,\ep}\Big)\hat T^\varepsilon  \phi\,dx_1dz
\\
\noame\displaystyle\qquad +D\int_{\omega\times Z}\nabla^\perp_{\eta_\ep,\ep}\hat w^\ep\cdot \nabla_{\eta_\ep,\ep}(\eta_\ep^{-2}\hat T^\varepsilon) \phi\,dx_1dz+    k \int_{\omega\times \hat \Gamma_1}G\, \phi\,dx_1d\sigma+ O_\ep,
\end{array}\end{equation}
where $O_\varepsilon$ tends to zero. Below, we pass to the limit in every terms:
\begin{itemize}
\item For the first two terms in the left-hand side of (\ref{form_var_1_changvar_hat3pass}), by using convergence (\ref{conv_T_crit_hat}), we get
$$ \eta_\ep^2\ep^{-2}\int_{\omega\times Z}  \eta_\ep^{-2}\partial_{z_1}\hat T^\varepsilon\, \partial_{z_1} \phi\,dx_1dz+\eta_\ep^{-2}\int_{\omega\times Z}  \partial_{z_2}\hat T^\varepsilon\, \partial_{z_2}  \phi\,dx_1dz\to \int_{\omega\times Z}\nabla_{\lambda} \hat T\cdot \nabla_{\lambda}  \phi\,dx_1dz.$$

\item For the first term in the right-hand side of (\ref{form_var_1_changvar_hat3pass}), by using estimates (\ref{estim_u_hat}) and (\ref{estim_T_hat}), we get
$$\left|- \eta_\ep\int_{\omega\times Z}\Big( \hat{\bf  u}^\varepsilon\cdot  \nabla_{\eta_\ep,\ep}\Big)\hat T^\varepsilon  \phi\,dx_1dz\right|\leq C\eta_\ep\| \hat {\bf u}^\varepsilon\|_{L^2(\omega\times Z)^2}\|\nabla_{\eta_\ep,\ep}\hat T^\varepsilon\|_{L^2(\omega\times Z)}\leq C\eta_\ep^3,$$
so we have
$$- \eta_\ep\int_{\omega\times Z}\Big(\hat{\bf  u}^\varepsilon\cdot  \nabla_{\eta_\ep,\ep}\Big)\hat T^\varepsilon  \phi\,dx_1dz\to 0.$$

\item For the second term in the right-hand side of (\ref{form_var_1_changvar_hat3pass}), by using convergences (\ref{conv_T_crit_hat}) and the strong convergence of $\nabla^\perp_{\eta_\ep,\ep}\hat w^\ep$ to $\nabla_\lambda^\perp \hat w$, we get
$$D\int_{\omega\times Z}\nabla^\perp_{\eta_\ep,\ep}\hat w^\ep\cdot \nabla_{\eta_\ep,\ep}(\eta_\ep^{-2}\hat T^\varepsilon) \phi\,dx_1dz
\to D \int_{\omega\times Z} \nabla^\perp_{\lambda}\hat w\cdot \nabla_{\lambda} \hat T \phi\,dx_1dz.$$
\end{itemize}

\noindent Then, using previous convergences, we get that the limit variational formulation for $\hat T$ is given by
\begin{equation}\label{form_var_1_changvar_proof3_limit1}
\begin{array}{l}
\displaystyle \int_{\omega\times Z}\nabla_{\lambda} \hat T\cdot \nabla_{\lambda}  \phi\,dx_1dz
=D\int_{\omega\times Z}\nabla_{\lambda}^\perp \hat w \cdot \nabla_{\lambda} \hat T  \phi\,dx_1dz+k \int_{\omega\times \hat \Gamma_1}G(z_1)  \phi\,dx_1d\sigma.
\end{array}\end{equation}
By density, (\ref{form_var_1_changvar_proof3_limit1}) holds for every function $\phi$ in $L^2(\omega;H^1_{\#}(Y))$ and is equivalent to problem (\ref{hom_system_crit_T}). We   remark that (\ref{form_var_1_changvar_proof3_limit1}) admits a unique solution, and then the complete sequence $\hat T^\ep$ converges to the unique solution $\hat T(x_1,z)$.

\end{proof}

\noindent Finally, we give the main result concerning the homogenized flow.

\begin{theorem}[Main result for the critical case] \label{mainthmcrit}
Consider $(\tilde{\bf  U}, \tilde W, \tilde \theta, \tilde p)$ given in Lemma \ref{lem_asymp_crit}. Let us define the average velocity, microrotation and temperature respectively by
$${\bf  U}^{av}(x_1)=\int_0^{h_{\rm max}}   \tilde{\bf  U}(x_1,z_2)\,dz_2,\quad W^{av}(x_1)=\int_0^{h_{\rm max}} \tilde W(x_1,z_2)\,dz_2,\quad T^{av}(x_1)=\int_0^{h_{\rm max}} \tilde \theta(x_1,z_2)\,dz_2.$$
\noindent We have the following:
\begin{itemize}
\item   The average velocity is given by
\begin{equation}\label{Velocity_thm_crit}
\begin{array}{lll}
\displaystyle U_1^{av}=a_\lambda{1-N\over Pr}\left(q_{-1/2}-q_{1/2}+Pr\int_{-1/2}^{1/2}f_1(\xi)\,d\xi\right),&  U_2^{av}(x_1)=0&\hbox{in }\omega,\end{array}
\end{equation}
where $a_\lambda \in \mathbb{R}$ is given by
$$a_\lambda=\int_{Z}|\nabla_\lambda{\bf u}^{bl}(z)|^2\,dz,$$
with $({\bf u}^{bl}, \pi^{bl}) \in H^1_\#(Z)^2\times L^2_{\#}(Z)$ the solution of the local Stokes problem
\begin{equation}\label{local_problems_critu}
\left\{\begin{array}{rl}\displaystyle
-\Delta_{\lambda} {\bf u}^{bl}+\nabla_{\lambda}\pi^{bl}={\bf e}_1&\hbox{ in }Z,\\
\noame
{\rm div}_{\lambda} {\bf u}^{bl}=0&\hbox{ in }Z,\\
\noame
u^{bl}=0&\hbox{ on }z_2=\{0,h(z_1)\},\\
\noame
\displaystyle\int_{Z}u_2^{bl}(z)dz=0.
\end{array}\right.
\end{equation}

\item  The pressure $\tilde p$ is given by
\begin{equation}\label{Pcritic}\tilde p(x_1) =q_{-1/2}-\left(q_{-1/2}-q_{1/2}+Pr\int_{-1/2}^{1/2}f_1(\xi)\,d\xi\right)\left(x_1+{1\over 2}\right)+Pr\int_{-1/2}^{x_1}f(\xi)\,d\xi\quad \hbox{in } \omega.
\end{equation}

 \item  The average   microrotation is given by
\begin{equation}\label{Micro_thm_crit}
 W^{av}(x_1)=b_\lambda{ 1\over L} g(x_1)\quad\hbox{in }\omega,
\end{equation}
where $b_\lambda\in \mathbb{R}$ is given by
$$ b_\lambda=\int_{Z}|\nabla_\lambda w^{bl}(z)|^2\,dz,$$
with $w^{bl} \in H^1_\#(Z)$ the solution of the local  Laplace problem
\begin{equation}\label{local_problems_critw}
\left\{\begin{array}{rl}\displaystyle
-\Delta_\lambda w^{bl}=1&\hbox{ in }Z,\\
\noame
w^{bl}=0&\hbox{ on }z_2=\{0,h(z_1)\}.
\end{array}\right.
\end{equation}

\item  The average temperature is given by
\begin{equation}\label{crit_average_temp_local}\displaystyle T^{av}(x_1)=\int_Z T^{bl}(x_1,z)\,dz\quad\hbox{in }\omega,
\end{equation}
with $ T^{bl}\in L^2(\omega;H^1_{\#}(Z))$ the unique solution of the nonlinear local problem
$$\left\{\begin{array}{rl} \displaystyle -\Delta_{\lambda}  T^{bl} -{D\over L}g(x_1)(\nabla_{\lambda}^\perp w^{bl}\cdot \nabla_{\lambda}) T^{bl}=0& \hbox{ in }\quad\omega\times Z,\\
\noame
 T^{bl}=0&\hbox{ on }\quad\omega\times \hat \Gamma_0 ,\\
\noame
\nabla_\lambda  T^{bl}\cdot {\bf n} =k\,G(z_1)&\hbox{ on }\quad\omega\times \hat \Gamma_1.
\end{array}\right.
$$

\end{itemize}
\end{theorem}

\begin{proof}
First, we proceed to  eliminate the microscopic variable $z$ in the effective linear problems (\ref{hom_system_crit_u}) and (\ref{hom_system_crit_w}). To do that, we consider the following identification
$$\begin{array}{l}
\hat{\bf  u}(x_1,z)= (1-N)\left(f_1(x_1)-{1\over Pr}\partial_{x_1} \tilde P(x_1)\right) {\bf u}^{bl}(z),\quad \hat q(x_1,z)=Pr \left(f_1(x_1)-{1\over Pr}\partial_{x_1}\tilde P(x_1)\right) \pi^{bl}(z), \\
\noame
\displaystyle \hat w(x_1,z)={g(x_1)\over L} w^{bl}(z),
\end{array}
$$
where $({\bf u}^{bl}, \pi^{bl})$ and $w^{bl}$ satisfies (\ref{local_problems_critu}) and (\ref{local_problems_critw}), respectively.\\

\noindent From the identities for velocity $\int_Z \hat u_1(x_1,z)\,dz=\int_0^{h_{\rm max}}\tilde U_1(x_1,z_2)\,dz_2$ and $\int_Z \hat u_2\,dz=0$, and for the microrotation $\int_Z \hat w(x_1,z)\,dz=\int_0^{h_{\rm max}}\tilde W(x_1,z_2)\,dz_2$ given in Lemma \ref{lem_asymp_crit}, by linearity we deduce that ${\bf   U}^{av}$  is given by
$$U_1^{av}=a_\lambda(1-N)\left(f_1(x_1)-{1\over Pr}\partial_{x_1}\tilde p(x_1)\right)\,\quad U_2^{av}=0,\quad \hbox{in }\omega,$$
 and $ W^{av}$ is given by (\ref{Micro_thm_crit}).\\

\noindent Next, the divergence condition with respect to the variable $x_1$ given in (\ref{div_x_crit_tilde}) together with the expression of $ U_1^{av}$ gives that \begin{equation}\label{Reynolds_crit}
\begin{array}{l}
U_1^{av}=a_\lambda(1-N)\left(f_1(x_1)-{1\over Pr}\partial_{x_1}\tilde p(x_1)\right)=C_1,\quad C_1\in\mathbb{R}.
\end{array}
\end{equation}
Then, integrating with respecto to $x_1$, and taking into account that $\tilde p(-1/2)=q_{-1/2}$,
 it holds
$$ \tilde p(x_1) =q_{-1/2}-{Pr\over a_\lambda(1-N)}C_1\left(x_1+{1\over 2}\right)+Pr\int_{-1/2}^{x_1}f(\xi)\,d\xi.$$
Finally, since $\tilde p(1/2)=q_{1/2}$, we deduce
$$C_1={a_\lambda(1-N)\over Pr}\left(q_{-1/2}-q_{1/2}+Pr\int_{-1/2}^{1/2}f_1(\xi)\,d\xi\right).$$
This  implies (\ref{Pcritic}). Then, by using the expression of $\tilde p$, we deduce that the average velocity $U_1^{av}$ is given by (\ref{Velocity_thm_crit}).\\

\noindent Finally, the formula for $ T^{av}$ follows from (\ref{Velocity_thm_crit})$_3$ and the identity $\int_Z \hat T(x_1,z)\,dz=\int_0^{h_{\rm max}}\tilde \theta(x_1,z_2)\,dz_2$ by renaming $T^{bl}\equiv \hat T$.

\end{proof}

\section{Homogenized model in the subcritical case}\label{sec:subcritical}
It corresponds to the case when the wavelength of the roughness is much greater than
the film thickness, i.e., $\eta_\ep\ll \ep$,  which is equivalent to $\lambda=0$.\\

\noindent We start by giving some compactness results about the behavior of the extended sequences $( \tilde{\bf  U}^\ep, \tilde W^\ep, \tilde \theta^\ep, p^\ep_0, \tilde p^\ep_1)$ and the related unfolding functions $( \hat{\bf  u}^\ep, \hat w^\ep, \hat T^\ep, \hat p^\ep_0, \hat p^\ep_1)$ satisfying the {\it a priori} estimates given in Lemmas \ref{lemma_estimates2} and Lemma \ref{estimates_hat} respectively.
\begin{lemma}\label{lem_asymp_sub}
For a subsequence of $\ep$ still denote by $\ep$, we have the following convergence results:
\begin{itemize}
\item[(i)] {\it (Velocity)} There exist $ \tilde{\bf  U}=(\tilde U_{1},\tilde U_{2})\in H^1(0,h_{\rm max};L^2(\omega)^2)$, with $ \tilde{\bf  U}=0$ on $z_2=\{0,h_{\rm max}\}$ and $\tilde U_{2}=0$, such that
\begin{eqnarray}
&\displaystyle \tilde{\bf  U}^\ep\rightharpoonup  \tilde{\bf  U}\quad\hbox{  in  }H^1(0,h_{\rm max};L^2(\omega)^2),\label{conv_u_sib_tilde}\\
\noame
&\displaystyle \partial_{x_1}\left(\int_0^{h_{\rm max}}\tilde U_1(x_1,y_2)\,dy_2\right)=0\quad \hbox{  in  }\omega. &\label{div_x_sub_tilde}
 \end{eqnarray}
and  $  \hat{\bf  u}=(\hat u_1, \hat u_{2})\in H^1(0,h(z_1);L^2_\#(\omega\times Z')^2)$, with $\hat {\bf u}=0$ on $z_2=\{0,h(z_1)\}$ and $\hat u_2=0$, such that it hold $\int_{Z} \hat{\bf  u}(x_1,z)dz$ $=\int_0^{h_{\rm max}} \tilde{\bf  U}(x_1,z_2)\,dz_2$ with $\int_{Z}\hat u_{2}(x_1,z)\,dz=0$, and moreover
\begin{eqnarray}
&
 \hat{\bf  u}^\ep\rightharpoonup  \hat {\bf u}\quad\hbox{  in  }H^1(0,h(z_1);L^2(\omega\times Z')^2),&\label{conv_u_sub_hat}\\
\noame
&\displaystyle \partial_{z_1}\left(\int_0^{h(z_1)}{ \hat u}_1\,dz_2\right)=0\quad\hbox{  in  }\omega\times Z',&\label{div_sub_hat1}\\
\noame
&\displaystyle \partial_{x_1}\left(\int_{Z}\hat u_1(x_1,z)\,dz\right)=0\quad \hbox{  in  }\omega\,.&\label{div_sub_hat2}
\end{eqnarray}
\item[(ii)] {\it (Microrotation)} There exist $\tilde W\in H^1(0,h_{\rm max};L^2(\omega))$, with $\tilde W=0$ on $z_2=\{0,h_{\rm max}\}$, such that
\begin{eqnarray}
&\displaystyle \tilde W^\ep\rightharpoonup \tilde W\quad\hbox{ in  }H^1(0,h_{\rm max};L^2(\omega)),\label{conv_w_sub_tilde}
\end{eqnarray}
and  $\hat w\in H^1(0,h(z_1);L^2_\#(\omega\times Z'))$, with $\hat w=0$ on $z_2=\{0,h(z_1)\}$ such that it hold $\int_{Z}\hat w(x_1,z)dz=\int_0^{h_{\rm max}}\tilde W(x_1,z_2)\,dz_2$, and moreover
\begin{eqnarray}
\hat w^\ep\rightharpoonup \hat w\quad \hbox{in } H^1(0,h(z_1);L^2(\omega\times Z')).&\label{conv_w_sub_hat}
\end{eqnarray}
\item[(iii)] {\it (Temperature)} There exist $\tilde \theta\in H^1(0,h_{\rm max};L^2(\omega))$, with $\tilde\theta=0$ on $z_2=\{0\}$, such that
\begin{eqnarray}
&\displaystyle \eta_\ep^{-{2}}\tilde \theta^\ep\rightharpoonup \tilde\theta\quad \hbox{in } H^1(0,h_{\rm max};L^2(\omega)),\label{conv_T_sub_tilde}
\end{eqnarray}
and  $\hat T\in H^1(0,h(z_1);L^2_\#(\omega\times Z'))$, with $\hat T=0$ on $z_2=\{0\}$, such that   $\int_{Z}\hat T(x_1,z)dz$ $=\int_0^{h_{\rm max}}\tilde \theta(x_1,z_2)\,dz_2$, and moreover
\begin{eqnarray}
\eta_\ep^{-2}\hat T^\ep\rightharpoonup \hat T\quad \hbox{in } H^1(0,h(z_1);L^2(\omega\times Z')).&\label{conv_T_sub_hat}
\end{eqnarray}
\item[(iv)] {\it (Pressure)} There  exist three functions  $\tilde p\in L^2_0(\omega)\cap H^1(\omega)$, independent of $z_2$ with with $\tilde p(i)=q_i$, $i=-1/2,1/2$, $\hat p_0\in L^2(\omega;H^1_\#(Z'))$ and $\hat p_1\in L^2(\omega;L^2_\#(Z))$ such that
\begin{eqnarray}
&\displaystyle\eta_\varepsilon^{2}  p^\ep_0\rightharpoonup  \tilde p\quad \hbox{in } H^1(\omega),& \label{conv_P01_sub1}\\
\noame
&\displaystyle \eta_\varepsilon^{2}\ep^{-1} \partial_{z_1}\hat p^\ep_0\rightharpoonup  \partial_{z_1}\tilde p+\partial_{z_1} \hat p_0\quad \hbox{in } L^2(\omega;L^2(Z')),\quad \eta_\ep \hat p^\ep_1\rightharpoonup  \hat p_1\quad \hbox{in } L^2(\omega;L^2(Z)).& \label{conv_P01_sub2}
\end{eqnarray}
\end{itemize}
\end{lemma}
\begin{proof}
Proof The proof of $(i)$ is similar to the critical case, but we have to take into account that
applying the unfolded change of variables  to the divergence condition ${\rm div}_{\eta_\ep}(\tilde u_\varepsilon)=0$ and multiplying by $\eta_\ep$, we get
\begin{equation}\label{div_proof}
{\eta_\varepsilon\over \varepsilon}\partial_{z_1}\hat u^\ep_{1} +\partial_{z_2}\hat u^\ep_{2}=0.
\end{equation}
Passing to the limit, since $\eta_\ep\ll\ep$, we get $\partial_{z_2}\hat u_2=0$, which means that $\hat u_2$ is independent of $z_2$. Due to the boundary conditions on the top and bottom, it holds that $\hat u_2=0$.  Now, multiplying (\ref{div_proof}) by $\varepsilon\eta_\ep^{-1}\varphi$ with $\varphi$ independent of $z_2$ and integrating by parts, we get
$$\int_{\omega\times Z'}\left(\int_0^{h(z_1)}\hat u^\ep_{1}\,dz_2\right)\partial_{z_1}\varphi\,dx_1dz_1=0.$$

\noindent Passing to the limit and integrating by parts, we get (\ref{div_sub_hat1}).
For more details, we refer the reader to the proof of Lemmas 5.2-i) and 5.4-ii) in \cite{Anguiano_SG} (see also \cite{grauMicRough}).
The proofs of $(ii)$, $(iii)$ and $(iv)$ are similar to the critical case, so we omit it.

\end{proof}

\noindent Using previous convergences, in the following theorem we give the two-pressured homogenized
system satisfied by $(\hat{\bf  u},\hat w, \tilde P, \hat T)$.

\begin{theorem}[Limit unfolded problems]\label{thm_sub1}
In the case $\eta_\ep\ll \ep$, then the functions $ \hat{\bf  u}, \hat w, \hat T$ and $\tilde p$ given in Lemma \ref{lem_asymp_sub} satisfy
\begin{itemize}
\item  $( \hat{\bf  u}, \tilde p)\in H^1(0,h(z_1);L^2_\#(\omega\times Z'))\times  (L^2_0(\omega)\cap H^1(\omega))$ with $\hat u_2=0$ is the unique solution of the two-pressure homogenized reduced Stokes problem
\begin{equation}\label{hom_system_sub_u}
\left\{\begin{array}{rl}
\displaystyle
-{1\over 1-N}\partial_{z_2}^2 { \hat u_1}+{1\over Pr}\partial_{z_1}\hat p_0=f_1(x_1)-{1\over Pr}\partial_{x_1}\tilde  p(x_1)&\hbox{ in }\omega\times Z,\\
\noame
\displaystyle \partial_{z_1}\left(\int_0^{h(z_1)}\hat u_1\,dz_2\right)=0&\hbox{ in }\omega\times Z',\\
\noame
{ \hat u_1}=0&\hbox{ on }\omega \times (\hat \Gamma_0\cup \hat \Gamma_1),\\
\displaystyle \partial_{x_1}\left(\int_{Z}\hat u_1(x_1,z)\,dz\right)=0&\hbox{ in }\omega,\\
\noame
\tilde p(i)=q_i&\ i=-1/2,1/2,\\
\noame
\hat p_0\in L^2_{\#}(\omega\times Z').
\end{array}\right.
\end{equation}

\item   $\hat w\in L^2(\omega;H^1_{\#}(Z))$ is the unique solution of the Laplace problem
\begin{equation}\label{hom_system_sub_w}
\left\{\begin{array}{rl}
\displaystyle
-L\partial_{z_2}^2 \hat w =g(x_1)&\hbox{ in }\omega\times Z,\\
\noame
\hat w=0&\hbox{ on }\omega \times (\hat \Gamma_0\cup \hat \Gamma_1),
\end{array}\right.
\end{equation}
\item   $\hat T\in L^2(\omega;H^1_{\#}(Z))$ is the unique solution of the nonlinear problem
\begin{equation}\label{hom_system_sub_T}
\left\{\begin{array}{rl}
 \partial_{z_2}^2 \hat T=0&\hbox{ in }\omega\times Z,\\
\noame
\hat T=0&\hbox{ on }z_2=\omega \times \hat \Gamma_0,\\
\noame
\partial_{z_2} \hat T=k\, G(z_1)&\hbox{ on }\omega \times \hat \Gamma_1.
\end{array}\right.
\end{equation}
\end{itemize}
\end{theorem}
\begin{proof} We divide the proof in three steps.\\

\noindent {\it Step 1.} To prove   (\ref{hom_system_sub_u})$_{1}$, we consider  in (\ref{form_var_1_changvar_hat1}) where $\varphi(x',z)\in \mathcal{D}(\omega;C^\infty_{\#}(Z)^2)$ with $\varphi_2=0$ in $\omega\times Z$. This gives the following variational formulation:
\begin{equation}\label{Form_var_sub_pass_limit}
 \begin{array}{l}
\displaystyle {1\over 1-N}\int_{\omega\times Z}\eta_\ep^2\ep^{-2} \partial_{z_1}{\hat u}^\varepsilon_1\, \partial_{z_1}  \varphi_1\,dx_1dz+
 {1\over 1-N}\int_{\omega\times Z}\partial_{z_2}{ \hat u_1}^\varepsilon\, \partial_{z_2} \varphi_1\,dx_1dz\\
 \noame
 \qquad\displaystyle
 +{1\over Pr}\int_{\omega\times Z}\eta_\ep^2\ep^{-1}\partial_{z_1}\hat p^\ep_0\, \varphi_1 \,dx_1dz-{1\over Pr}\int_{\omega\times Z}\eta_\ep^2\ep^{-1}\hat p^\varepsilon_1\,\partial_{z_1}\, \varphi_1\,dx_1dz
\\
\noame
\displaystyle\qquad 
={\eta_\ep^2\varepsilon^{-1}\over Pr}\int_{\omega\times Z}{\hat  u}^\varepsilon_1  {\hat  u}^\varepsilon_1   \partial_{z_1}\varphi_1\,dx_1dz-{\eta_\varepsilon \over Pr}\left(\int_{\omega\times Z}\partial_{z_2}{\hat u}^\varepsilon_2{  \hat u}^\varepsilon_1 \varphi_1 \,dx_1dz+
\int_{\omega\times Z} \hat u^\varepsilon_2\partial_{z_2}{  \hat u}^\varepsilon_1 \varphi_1 \,dx_1dz\right)
\\
\noame
\displaystyle\qquad

+{2N\over 1-N}\int_{\omega\times Z} \eta_\varepsilon\partial_{z_2} \hat w^\ep \varphi_1 \,dx_1dz
+ \int_{\omega\times Z}f_1 \,\varphi_1\,dx_1dz+O_\ep,
\end{array}
\end{equation}
where $O_\varepsilon$ is devoted to tends to zero when $\varepsilon\to 0$. Below, let us pass to the limit when $\varepsilon$ tends to zero in each term of the previous variational formulation:
\begin{itemize}
\item For the first two terms in the left-hand side of (\ref{Form_var_sub_pass_limit}), taking into account convergence (\ref{conv_u_sub_hat}) and that $\eta_\ep/\ep\to 0$, we get that
$$\begin{array}{l}
\displaystyle
 {1\over 1-N}\int_{\omega\times Z}\eta_\ep^2\ep^{-2} \partial_{z_1}{\hat u}^\varepsilon_1\, \partial_{z_1}  \varphi_1\,dx_1dz\to 0,\\
 \noame
 \displaystyle
{1\over 1-N}\int_{\omega\times Z}\partial_{z_2}{ \hat u_1}^\varepsilon\, \partial_{z_2} \varphi_1\,dx_1dz\to {1\over 1-N}\int_{\omega\times Z}\partial_{z_2}{ \hat u_1}\, \partial_{z_2} \varphi_1\,dx_1dz.
\end{array}
$$

\item For the third term on the left hand side of (\ref{Form_var_sub_pass_limit}), taking into account that  convergence of the pressures (\ref{conv_P01_sub2}) and  $\eta_\ep/\ep\to 0$, we have the following convergence s
$$\begin{array}{l}
\displaystyle
{1\over Pr}\int_{\omega\times Z}\eta_\ep^2\ep^{-1}\partial_{z_1}\hat p^\ep_0\, \varphi_1 \,dx_1dz\to {1\over Pr}\int_{\omega\times Z}(\partial_{x_1}\tilde p+\partial_{z_1}\hat p_0)\, \varphi_1 \,dx_1dz,\\
\noame
\displaystyle
-{1\over Pr}\int_{\omega\times Z}\eta_\ep^2\ep^{-1}\hat p^\varepsilon_1\,\partial_{z_1}\, \varphi_1\,dx_1dz\to 0.
\end{array}$$

\item For the first three terms in the right-hand side of (\ref{Form_var_sub_pass_limit}), by taking into account the estimates (\ref{estim_u_hat}), we get
$$\begin{array}{l}\displaystyle
\left|{\eta_\ep^2\varepsilon^{-1}\over Pr}\int_{\omega\times Z}{\hat  u}^\varepsilon_1  {\hat  u}^\varepsilon_1   \partial_{z_1}\varphi_1\,dx_1dz\right|\\
\noame\displaystyle\qquad \leq {\eta_\ep^2\varepsilon^{-1}}\|\hat{\bf  u}^\varepsilon\|^2_{L^2(\omega\times Z)^2}\|\partial_{z_1}\varphi\|_{L^\infty(\omega\times Z)^2}\leq C{\eta_\ep^2\varepsilon^{-1}}\to 0,
\end{array}$$

$$\begin{array}{l}\displaystyle
\left|-{\eta_\varepsilon \over Pr}\left(\int_{\omega\times Z}\partial_{z_2}{\hat u}^\varepsilon_2{  \hat u}^\varepsilon_1 \varphi_1 \,dx_1dz+
\int_{\omega\times Z} \hat u^\varepsilon_2\partial_{z_2}{  \hat u}^\varepsilon_1 \varphi_1 \,dx_1dz\right)
\right|\\
\noame \displaystyle\qquad \leq \eta_\varepsilon\|\hat{\bf  u}^\varepsilon\|_{L^2(\omega\times Z)^2}\|\partial_{z_2} \hat{\bf  u}^\varepsilon\|_{L^2(\omega\times Z)^2}\|\varphi\|_{L^\infty(\omega\times Z)^2}\\
\noame
\qquad \leq  C\eta_\varepsilon\to 0.
\end{array}
$$
Then, we deduce that the convective terms satisfy
$${\eta_\ep^2\varepsilon^{-1}\over Pr}\int_{\omega\times Z}{\hat  u}^\varepsilon_1  {\hat  u}^\varepsilon_1   \partial_{z_1}\varphi_1\,dx_1dz-{\eta_\varepsilon \over Pr}\left(\int_{\omega\times Z}\partial_{z_2}{\hat u}^\varepsilon_2{  \hat u}^\varepsilon_1 \varphi_1 \,dx_1dz+
\int_{\omega\times Z} \hat u^\varepsilon_2\partial_{z_2}{  \hat u}^\varepsilon_1 \varphi_1 \,dx_1dz\right)\to 0.$$

\item For the fourth term in the right-hand side of (\ref{Form_var_sub_pass_limit}), by taking into account convergence (\ref{conv_w_sub_hat}), so we have
$${2N\over 1-N}\int_{\omega\times Z} \eta_\varepsilon\partial_{z_2} \hat w^\ep \varphi_1 \,dx_1dz\to 0.$$
\end{itemize}
Therefore, by previous convergences, we deduce that the limit variational formulation is given by the following one
\begin{equation}\label{form_var_1_changvar_hat1_limit2_sub2}
 \begin{array}{l}
\displaystyle {1\over 1-N}\int_{\omega\times Z}\partial_{z_2}{ \hat u_1}\, \partial_{z_2} \varphi_1\,dx_1dz+{1\over Pr}\int_{\omega\times Z}\partial_{x_1}\tilde p\, \varphi_1 \,dx_1dz+{1\over Pr}\int_{\omega\times Z}\partial_{z_1}\hat p_0\, \varphi_1 \,dx_1dz =
 \int_{\omega\times Z}f_1\, \varphi_1\,dx_1dz.
\end{array}\end{equation}

\noindent By density, (\ref{form_var_1_changvar_hat1_limit2_sub2}) holds for every function $\varphi$ in the $H^1(0,h(z_1);L^2_\#(\omega\times Z'))$ and is equivalent to problem  (\ref{hom_system_sub_u})$_{1}$.  We  remark that (\ref{form_var_1_changvar_hat1_limit2_sub2}) admits a unique solution, and then the complete sequences converge. \\

 \noindent {\it Step 2.} Next, we  prove that $\hat w$ satisfies problem (\ref{hom_system_sub_w}).  Below, let us pass to the limit when $\varepsilon$ tends to zero in each term of the previous variational formulation (\ref{form_var_1_changvar_hat2}):
\begin{itemize}
\item For the first two terms in the left-hand side of (\ref{form_var_1_changvar_hat2}), by using convergence (\ref{conv_w_crit_hat}) and $\eta_\varepsilon/\ep\to 0$, we get
$$\begin{array}{l}
\displaystyle L\int_{\omega\times Z} \eta_\ep^2\ep^{-2}\partial_{z_1} \hat w^\varepsilon\, \partial_{z_1} \psi\,dx_1dz\to 0,\\
\noame\displaystyle
L\int_{\omega\times Z}  \partial_{z_2} \hat w^\varepsilon\, \partial_{z_2} \psi\,dx_1dz \to  L\int_{\omega\times Z}  \partial_{z_2} \hat w^\varepsilon\cdot \partial_{z_2} \psi\,dx_1dz,
\end{array}$$
and so,
$$L\int_{\omega\times Z} \eta_\ep^2\ep^{-2}\partial_{z_1} \hat w^\varepsilon\, \partial_{z_1} \psi\,dx_1dz+L\int_{\omega\times Z}  \partial_{z_2} \hat w^\varepsilon\, \partial_{z_2} \psi\,dx_1dz\to  L\int_{\omega\times Z}  \partial_{z_2} \hat w\, \partial_{z_2} \psi\,dx_1dz.$$
\item For the third term of the left-hand side of  (\ref{form_var_1_changvar_hat2}),   by using convergence (\ref{conv_w_crit_hat}), we have
$${4N\over 1-N}\int_{\omega\times Z}\eta_\ep^2 \hat w^\varepsilon  \psi\,dx_1dz\to 0.$$
\item For the first two terms in the right-hand side of  (\ref{form_var_1_changvar_hat2}), by using estimates (\ref{estim_u_hat}) and (\ref{estim_w_hat}), we get
$$\begin{array}{l}\displaystyle \left|-{M\over Pr}\int_{\omega\times Z} \eta_\ep^2\ep^{-1} \hat u^\varepsilon_1 \partial_{z_1}\hat w^\ep   \psi\,dx_1dz\right|\leq C\eta_\ep^2\ep^{-1}\|\hat{\bf  u^\varepsilon}\|_{L^2(\omega\times Z)^2}\|\partial_{z_1}\hat w^\ep\|_{L^2(\omega\times Z)}\leq C\eta_\ep,
\\
\\
\displaystyle
\left|-{M\over Pr}\int_{\omega\times Z} \eta_\ep \hat u^\varepsilon_2\partial_{z_2}\hat w^\ep   \psi\,dx_1dz\right|\leq
C\eta_\ep\|  \hat{\bf  u^\varepsilon}\|_{L^2(\omega\times Z)^2}\|\partial_{z_2}\hat w^\ep\|_{L^2(\omega\times Z)}\leq C\eta_\ep.
\end{array}$$
Thus, we get
$$-{M\over Pr}\int_{\omega\times Z} \eta_\ep^2\ep^{-1} \hat u^\varepsilon_1 \partial_{z_1}\hat w^\ep   \psi\,dx_1dz-{M\over Pr}\int_{\omega\times Z} \eta_\ep \hat u^\varepsilon_2\partial_{z_2}\hat w^\ep   \psi\,dx_1dz\to 0.$$
\item For the third and fourth terms in the right-hand side of (\ref{form_var_1_changvar_hat2}), by using estimates (\ref{estim_u_hat}) and (\ref{estim_w_hat}), we get
$$\begin{array}{l}\displaystyle
\left|{2N\over 1-N}\int_{\omega\times Z} \, \eta_\ep^2\ep^{-1}\partial_{z_1}\hat u^\varepsilon_2\, \psi\,dx_1dz\right|\leq C\eta_\ep^2\ep^{-1}\|\partial_{z_1}\hat{\bf  u}^\varepsilon\|_{L^2(\omega\times Z)^2}\leq C\eta_\ep,\\
\noame
\displaystyle
\left|-{2N\over 1-N}\int_{\omega\times Z} \, \eta_\ep\partial_{z_2}\hat u^\varepsilon_1\psi\,dx_1dz\right|\leq C\eta_\ep\|\partial_{z_2} \hat{\bf  u}^\varepsilon\|_{L^2(\omega\times Z)^2}\leq C\eta_\ep.
\end{array}$$
Thus, we have
$${2N\over 1-N}\int_{\omega\times Z} \, \eta_\ep^2\ep^{-1}\partial_{z_1}\hat u^\varepsilon_2\, \psi\,dx_1dz-{2N\over 1-N}\int_{\omega\times Z} \, \eta_\ep\partial_{z_2}\hat u^\varepsilon_1\psi\,dx_1dz\to 0.$$
\end{itemize}
Then, from the above convergences, we get that the limit variational formulation for $\hat w$ is given by
\begin{equation}\label{form_var_1_changvar_proof2_limit}\begin{array}{l}
\displaystyle L\int_{\omega\times Z}  \partial_{z_2} \hat w\, \partial_{z_2} \psi\,dx_1dz  =\int_{\omega\times Z}g\, \psi\,dx_1dz.
\end{array}\end{equation}
By density (\ref{form_var_1_changvar_proof2_limit}) holds for every function $\psi$ in $H^1(\omega;L^2_{\#}(\omega\times Z'))$ and is equivalent to problem (\ref{hom_system_sub_w})$_1$. We   remark that (\ref{form_var_1_changvar_proof2_limit}) admits a unique solution, and then the complete sequence converges.
\\

\noindent {\it Step 3}. Next, we  prove that $\hat T$ satisfies problem (\ref{hom_system_sub_T}). we take into account that the variational formulation (\ref{form_var_1_changvar_hat3}) can be written as follows
\begin{equation}\label{form_var_1_changvar_hat3pass}
\begin{array}{l}
\displaystyle \eta_\ep^2\ep^{-2}\int_{\omega\times Z}  \eta_\ep^{-2}\partial_{z_1}\hat T^\varepsilon\, \partial_{z_1} \phi\,dx_1dz+\eta_\ep^{-2}\int_{\omega\times Z}  \partial_{z_2}\hat T^\varepsilon\cdot \partial_{z_2}  \phi\,dx_1dz\\
\noame
\qquad\displaystyle
=- \eta_\ep\int_{\omega\times Z}\Big(\hat{\bf  u}^\varepsilon\cdot  \nabla_{\eta_\ep,\ep}\Big)\hat T^\varepsilon  \phi\,dx_1dz
\\
\noame\displaystyle\qquad +D\int_{\omega\times Z}\nabla^\perp_{\eta_\ep,\ep}\hat w^\ep\cdot \nabla_{\eta_\ep,\ep}(\eta_\ep^{-2}\hat T^\varepsilon) \phi\,dx_1dz+    k \int_{\omega\times \hat \Gamma_1}G\, \phi\,dx_1d\sigma+ O_\ep,
\end{array}\end{equation}
where we use the operators $\nabla_{\eta_\ep,\ep} =(\eta_\ep\ep^{-1} \partial_{z_1}, \partial_{z_2})$ and $\nabla^\perp_{\eta_\ep,\ep}=(\partial_{z_2},-\eta_\ep\ep^{-1} \partial_{z_1})$. Below, we pass to the limit in every terms:
\begin{itemize}
\item For the first two terms in the left-hand side of (\ref{form_var_1_changvar_hat3}), by using convergence (\ref{conv_T_sub_hat}) and $\eta_\varepsilon/\ep\to 0$, we get
$$
\begin{array}{l}
\displaystyle
\eta_\ep^2\ep^{-2}\int_{\omega\times Z}  \eta_\ep^{-2}\partial_{z_1}\hat T^\varepsilon\, \partial_{z_1} \phi\,dx_1dz\to 0,\\
\noame
\displaystyle
\eta_\ep^{-2}\int_{\omega\times Z}  \partial_{z_2}\hat T^\varepsilon\, \partial_{z_2}  \phi\,dx_1dz\to  \int_{\omega\times Z}  \partial_{z_2}\hat T \, \partial_{z_2}  \phi\,dx_1dz,
\end{array}$$
and so,
$$\eta_\ep^2\ep^{-2}\int_{\omega\times Z}  \eta_\ep^{-2}\partial_{z_1}\hat T^\varepsilon\, \partial_{z_1} \phi\,dx_1dz+\eta_\ep^{-2}\int_{\omega\times Z}  \partial_{z_2}\hat T^\varepsilon\, \partial_{z_2}  \phi\,dx_1dz\to  \int_{\omega\times Z}  \partial_{z_2}\hat T \, \partial_{z_2}  \phi\,dx_1dz.$$
\item For the first term in the left-hand side of (\ref{form_var_1_changvar_hat3}), by using estimates (\ref{estim_u_hat}) and (\ref{estim_T_hat}), we get
$$\left|- \eta_\ep\int_{\omega\times Z}\Big(\hat{\bf  u}^\varepsilon\cdot  \nabla_{\eta_\ep,\ep}\Big)\hat T^\varepsilon  \phi\,dx_1dz\right|\leq C\eta_\ep\| \hat{\bf  u}^\varepsilon\|_{L^2(\omega\times Z)^2}\|\nabla_{\eta_\ep,\ep}\hat T^\varepsilon\|_{L^2(\omega\times Z)}\leq C\eta_\ep^3,$$
so we have
$$- \eta_\ep\int_{\omega\times Z}\Big(\hat{\bf  u}^\varepsilon\cdot  \nabla_{\eta_\ep,\ep}\Big)\hat T^\varepsilon  \phi\,dx_1dz\to 0.$$
\item For the second term in the right-hand side of (\ref{form_var_1_changvar_hat3}), by using convergences (\ref{conv_T_sub_hat}), the strong convergence of $\nabla^\perp_{\eta_\ep,\ep}\hat w^\ep$ to $(\partial_{z_2}\hat w,0)$ (it can be proved as in the critical case) and the weak convergence of $\nabla_{\eta_\ep,\ep}(\eta_\ep^{-2}\hat T^\varepsilon)$ to $(0,\partial_{z_2}\hat T)$,  we get
$$D\int_{\omega\times Z}\nabla^\perp_{\eta_\ep,\ep}\hat w^\ep\cdot \nabla_{\eta_\ep,\ep}(\eta_\ep^{-2}\hat T^\varepsilon) \phi\,dx_1dz
\to 0.$$
\end{itemize}
\noindent Then, using previous convergences, we get that the limit variational formulation for $\hat T$ is given by
\begin{equation}\label{form_var_1_changvar_proof3_limit1_sub}
\begin{array}{l}
\displaystyle \int_{\omega\times Z}\partial_{z_2} \hat T\,\partial_{z_2} \phi\,dx_1dz
=k \int_{\omega\times \hat \Gamma_1}G(z_1)  \phi\,dx_1d\sigma.
\end{array}\end{equation}
By density (\ref{form_var_1_changvar_proof3_limit1}) holds for every function $\phi$ in $H^1(\omega;L^2_{\#}(\omega\times Z'))$. We   remark that (\ref{form_var_1_changvar_proof3_limit1}) admits a unique solution, and then the complete sequence converges.

\end{proof}

\noindent Finally, we give the main result concerning the homogenized flow.

\begin{theorem}[Main result for the subcritical case] \label{mainthmsubcrit}
Consider $(\tilde{\bf  U}, \tilde W, \tilde \theta, \tilde p)$ given in Lemma \ref{lem_asymp_sub}. Let us define the average velocity, microrotation and temperature respectively by
$${\bf  U}^{av}(x_1)=\int_0^{h_{\rm max}}   \tilde {\bf  U}(x_1,z_2)\,dz_2,\quad W^{av}(x_1)=\int_0^{h_{\rm max}} \tilde W(x_1,z_2)\,dz_2,\quad T^{av}(x_1)=\int_0^{h_{\rm max}} \tilde \theta(x_1,z_2)\,dz_2.$$
\noindent We have the following:
\begin{itemize}
\item   The average velocity is given by
\begin{equation}\label{Velocity_thm_sub}
\begin{array}{lll}
\displaystyle U_1^{av}={a_0}{1-N\over Pr}\left(q_{-1/2}-q_{1/2}+Pr\int_{-1/2}^{1/2}f_1(\xi)\,d\xi\right),&  U_2^{av}=0&\hbox{in }\omega,\end{array}
\end{equation}
where $a_0 \in \mathbb{R}$ is given by
\begin{equation}\label{Velocity_thm_sub2}
a_0={1\over 12}\int_{-1/2}^{1/2}h^3(z_1)\left(2-h^3(z_1)\left(\int_{-1/2}^{1/2}h^3(\xi)\,d\xi\right)^{-1}\right)\,dz_1.
\end{equation}
\item The pressure $\tilde p$ is given by
\begin{equation}\label{Psub}\tilde p(x_1) =q_{-1/2}-\left(q_{-1/2}-q_{1/2}+Pr\int_{-1/2}^{1/2}f_1(\xi)\,d\xi\right)\left(x_1+{1\over 2}\right)+Pr\int_{-1/2}^{x_1}f_1(\xi)\,d\xi\quad \hbox{in } \omega.
\end{equation}

 \item  The average   microrotation is given by
\begin{equation}\label{Micro_thm_sub}
 W^{av}(x_1)={b_0 }{ 1\over L} g(x_1)\quad\hbox{in }\omega,
\end{equation}
where $b_0\in \mathbb{R}$ is given by
\begin{equation}\label{Micro_thm_sub2}b_0={1\over 12}\int_{-1/2}^{1/2}h^3(z_1)\,dz_1.
\end{equation}

\item  The average temperature is given by
\begin{equation}\label{expTav_sub1}\displaystyle T^{av}=c_0{k}\quad\hbox{in }\omega,
\end{equation}
where $c_0\in\mathbb{R}$ is given by
\begin{equation}\label{expTav_sub2}c_0={1\over 2}\int_{-1/2}^{1/2}h^2(z_1)G(z_1)\,dz_1.
\end{equation}

\end{itemize}
\end{theorem}

\begin{proof}
First, we start with the velocity by proceeding to  eliminate the microscopic variable $z$ in the effective linear problem (\ref{hom_system_sub_u}). To do that, as in the critical case, we consider the following identification
$$\begin{array}{l}
{\hat u}_1(x_1,z)= -(1-N)\left(f_1(x_1)-{1\over Pr}\partial_{x_1} \tilde p(x_1)\right) {u}^{bl}(z),\quad \hat p_0(x_1,z)= -Pr\left(f_1(x_1)-{1\over Pr}\partial_{x_1}\tilde p(x_1)\right) \pi^{bl}(z).
\end{array}
$$
From the identities for velocity $U_1^{av}=\int_0^{h_{\rm max}}\tilde U_1(x_1,z_2)\,dz_2=\int_Z \hat u_1(x_1,z)\,dz$ and $ \hat u_2=0$ given in Lemma \ref{lem_asymp_sub}, by linearity we deduce that ${\bf   U}^{av}$  is given by
\begin{equation}\label{exp_uav_sub1}\begin{array}{l}\displaystyle
U_1^{av}=-a_0(1-N)\left(f_1(x_1)-{1\over Pr}\partial_{x_1}\tilde p(x_1)\right)\,\quad U_2^{av}=0,\quad \hbox{in }\omega.
\end{array}
\end{equation}
with $a_0$  given by
$$a_0=\int_{Z}u_1^{bl}\,dz,$$
where $({u}^{bl},\pi^{bl})$  satisfies the following local reduced problem
\begin{equation}\label{hom_system_sub_u_local}
\left\{\begin{array}{rl}
\displaystyle
- \partial_{z_2}^2 { u}^{bl}+\partial_{z_1}\pi^{bl}=-1&\hbox{ in }\omega\times Z,\\
\noame
\displaystyle \partial_{z_1}\left(\int_0^{h(z_1)}u^{bl}\,dz_2\right)=0&\hbox{ in }\omega\times Z',\\
\noame
{ u}^{bl}=0&\hbox{ on }\omega \times (\hat \Gamma_0\cup \hat \Gamma_1),\\
\displaystyle \partial_{x_1}\left(\int_{Z} u^{bl}\,dz\right)=0&\hbox{ in }\omega.\\
\end{array}\right.
\end{equation}

\noindent Now, we observe that we can obtain more accurate expressions for $a_0$, because problem (\ref{hom_system_sub_u_local}) is an ordinary differential equation with respect to the variable $z_2$ and it can be solved. Thus, from the boundary conditions on the top and bottom, we get
\begin{equation}\label{exp_uav_sub3}\begin{array}{l}
\displaystyle u^{bl}(z)={1\over 2}\left(1+ \partial_{z_1}\pi^{bl}\right)\left(z_2^2-h(z_1)z_2\right).
\end{array}
\end{equation}
Taking into account that $\int_0^{h(z_1)}u^{bl}(z)\,dz_2=-{h(z_1)^3}(1+\partial_{z_1}\pi^{bl}(z_1))/12$ and the expression of $a_0$, we get
\begin{equation}\label{exp_uav_sub3}a_0=-{1\over 12}\int_{Z'}h^3(z_1)\left(1+\partial_{z_1}\pi^{bl}(z_1)\right)\,dz_1,
\end{equation}
where, by using (\ref{hom_system_sub_u})$_2$, then $\pi^{bl}\in L^2_\#(Z')/\mathbb{R}$ is the solution of the  second order ordinary differential equation with respect to $z_1$ with periodic boundary conditions on $Z'$, given by
\begin{equation}\label{localpbl}\left\{\begin{array}{l}
\displaystyle h^3(z_1)\partial_{z_1}^2\pi^{bl}(z_1)-3h^2(z_1){dh\over dz_1}(z_1)\partial_{z_1}\pi^{bl}(z_1)=-3h^2(z_1){dh\over dz_1}(z_1)\quad\hbox{in }Z',\\
\noame
\pi^{bl}(-1/2)=\pi^{bl}(1/2).
\end{array}\right.
\end{equation}
Solving this equation, we obtain an expression for $\pi^{bl}$, up to a constant,
$$\pi^{bl}(z_1)=-\left(\int_{-1/2}^{1/2}h^3(\xi)\,d\xi\right)^{-1}\int_{-1/2}^{z_1}h^3(\xi)\,d\xi+ z_1+1/2+C,\quad C\in\mathbb{R},\quad z_1\in Z'.$$
This implies that
$$\partial_{z_1}\pi^{bl}(z_1)=-\left(\int_{-1/2}^{1/2}h^3(\xi)\,d\xi\right)^{-1} h^3(z_1)+1,\quad z_1\in Z',$$
and so, from (\ref{exp_uav_sub3}), we get
\begin{equation}\label{a0sub}
a_0=-{1\over 12}\int_{-1/2}^{1/2}h^3(z_1)\left(2-h^3(z_1)\left(\int_{-1/2}^{1/2}h^3(\xi)\,d\xi\right)^{-1}\right)\,dz_1.
\end{equation}
From condition (\ref{div_x_sub_tilde}), by taking into account the expression of $U_1^{av}$ and the boundary conditions of $\tilde p$, we get the expression for pressure $\tilde p$ given in (\ref{Psub}).\\

\noindent Finally, taking into account the expressions of (\ref{Psub}), (\ref{exp_uav_sub1}) and (\ref{a0sub}), then the average velocity can be written as (\ref{Velocity_thm_sub})-(\ref{Velocity_thm_sub2}). \\

\noindent Next, we focus on the microrotation. We eliminate the microscopic variable $z$ in the effective linear problem (\ref{hom_system_sub_w}). To do that, as in the critical case, we consider the following identification
$$\displaystyle \hat w(x_1,z)={g(x_1)\over L} w^{bl}(z),$$
where $w^{bl}\in H^1_{\#}(Z)$ is the solution of the local problem
\begin{equation}\label{hom_system_sub_wbk_local}
\left\{\begin{array}{rl}
\displaystyle
-\partial_{z_2}^2 w^{bl} =-1&\hbox{ in } Z,\\
\noame
\hat w=0&\hbox{ on } \hat \Gamma_0\cup \hat \Gamma_1.
\end{array}\right.
\end{equation}
This implies that
$$w^{bl}(z)=-{1\over 2}\left(z_2^2-h(z_1)z_2\right),$$
and taking into account that $\int_0^{h(z_1)}w^{bl}\,dz_2=h^3(z_1)/12$ and that $W^{av}(x_1)=\int_{Z}\hat w\,dz$, we get
$$W^{av}={b_0}{g(x_1)\over L},\quad b_0={1\over 12}\int_{Z'}h^3(z_1)\,dz_1,$$
which is (\ref{Micro_thm_sub})-(\ref{Micro_thm_sub2}).\\

\noindent Finally,  we obtain the expression of the average of the temperature. To do this, we solve the problem (\ref{hom_system_sub_T}), which gives the expression for $\hat T$
$$\hat T(x_1,z)=k G(z_1)z_2,\quad\hbox{in } \omega\times Z.$$
Taking into account that $T^{av}(x_1)=\int_{Z}\hat T\,dz$, we easily get (\ref{expTav_sub1})-(\ref{expTav_sub2}).

\end{proof}

\ \\
\textbf{Acknowledgements.}\quad The first author of this paper has been supported by the {\emph {Croatian Science Foundation}} under the project Multiscale problems in fluid mechanics - MultiFM (IP-2019-04-1140).
\ \\

\end{document}